\documentclass[reqno]{amsart}
\usepackage[utf8]{inputenc}
\usepackage{amssymb,amsthm,amsfonts}
\usepackage{amsmath}
\usepackage{tikz}
\usepackage{graphicx}
\usepackage{psfrag}
\usepackage{cancel}
\usepackage{hyperref}
\newtheorem{definition}{Definition}[section]
\newtheorem{theorem}{Theorem}[section]
\newtheorem{lemma}{Lemma}[section]
\newtheorem{proposition}{Proposition}[section]
\newtheorem{corollary}{Corollary}[section]
\newtheorem{remark}{Remark}[section]
\newtheorem{conjecture}{Conjecture}[section]
\numberwithin{equation}{section}

\newtheorem{innercustomthm}{Cross-Ratio Inequality\!}
\newenvironment{customthm}[1]
  {\renewcommand\theinnercustomthm{#1}\innercustomthm}
  {\endinnercustomthm}

\newtheorem{innernewcustomthm}{Power Law\!\!}
\newenvironment{newcustomthm}[1]
  {\renewcommand\theinnernewcustomthm{#1}\innernewcustomthm}
  {\endinnernewcustomthm}

\newtheorem{innercustomthmA}{Theorem}
\newenvironment{customthmA}[1]
  {\renewcommand\theinnercustomthmA{#1}\innercustomthmA}
  {\endinnernewcustomthm}

\newcommand{\doublint}{\int\!\!\!\int}

\title[Real bounds and qs-rigidity of multicritical circle maps]{Real bounds and quasisymmetric rigidity of multicritical circle maps}
\author{Gabriela Estevez and Edson de Faria}


\address{Instituto de Matem\'atica e Estat\'istica, Universidade de S\~ao Paulo}
\curraddr{Rua do Mat\~ao 1010, 05508-090, S\~ao Paulo SP, Brasil}
\email{gestevez@ime.usp.br}
\email{edson@ime.usp.br}
 
\thanks{This work has been supported by ``Projeto Tem\'atico Din\^amica em Baixas Dimens\~oes'' 
FAPESP Grant 2011/16265-2, and by CAPES (PROEX)}

\begin{document}

\maketitle

\begin{abstract}
Let $f, g:S^1\to S^1$ be two $C^3$ critical homeomorphisms of the circle with the same irrational rotation number and the 
same (finite) number of critical points, all of which are assumed to be non-flat, of power-law type. 
In this paper we prove that if $h:S^1\to S^1$ 
is a topological conjugacy between $f$ and $g$ and $h$ maps the critical points of $f$ to the critical points of $g$, then 
$h$ is quasisymmetric. When the power-law exponents at all critical points are integers, 
this result is a special case of a general theorem recently proved by T.~Clark and S.~van Strien \cite{CS}. However, unlike 
the proof given in \cite{CS}, which relies on heavy complex-analytic machinery, our proof uses purely real-variable methods, and is valid 
for non-integer critical exponents as well. We do not require $h$ to preserve the power-law exponents at corresponding critical points.  
\end{abstract}

\section{Introduction}\label{sec:intro}

In dynamics, a {\it rigidity theorem\/} is one in which a {\it weak\/} equivalence between two systems implies, under a minimum set of hypotheses, 
a {\it strong\/} equivalence between those systems. For instance, a topological conjugacy (or perhaps even a combinatorial equivalence) 
between two systems may be proven in fact to be a differentiable conjugacy, provided such systems are smooth enough. 

In the context of homeomorphisms of the circle, there are well-developed rigidity theories for smooth diffeomorphisms 
(see \cite{H}, \cite{Y}), as well as for smooth circle homeomorphisms with exactly one non-degenerate critical point 
(see \cite{dFdM, dFdM2, KT, KY} for the real-analytic case, and \cite{Gua,GdM} for the case of $C^3$ homeos). 
By contrast, the theory for circle homeomorphisms with two or more critical points is far from being well-developed. 

For smooth one-dimensional systems,  there is oftentimes a close relationship between rigidity and {\it renormalization convergence\/}. 
The {\it first renormalization\/} $Rf$ of a map $f$ around some special point in phase space (usually a critical point) is given by the first return map 
of $f$ to a certain neighborhood of that special point, suitably rescaled. Renormalization can be seen as a (non-linear) operator acting on an 
appropriate space of such maps. In particular, some maps may be {\it infinitely renormalizable\/}, in the sense that the successive 
renomalizations $R^nf=R(R^{n-1}f)$ ($n\geq 1$) are well-defined. It so happens that, under suitable hypotheses, if two topologically conjugate maps 
$f$ and $g$ are infinitely renormalizable, then the $C^0$ distance between their successive renormalizations converges to zero. 
When this is the case, the general {\it ansatz\/} is that an exponential rate of convergence should yield a smooth conjugacy between $f$ and $g$.  

In the presence of critical points, say for real-analytic systems, the strategy towards rigidity usually 
involves (a variation of) the following steps.
\begin{enumerate}
 \item[(1)] Get {\it real a-priori bounds\/}: geometric bounds on the critical orbits.
 \item[(2)] Use the real bounds in (1) to promote the topological conjugacy between the two systems 
to a conjugacy with some mild geometric control (such as {\it quasisymmetry\/}, see \S \ref{sec:quasisym} below).
 \item[(3)] Complexify the given real dynamical systems and use the real bounds in (1) and the geometric 
control in (2) in order to get {\it complex a-priori bounds\/} for the complexified systems. These bounds 
(usually bounds on moduli of certain annuli) yield a strong form of compactness.
 \item[(4)] Extend the renormalization operator to the complexified dynamical systems. 
 \item[(5)] Use the bounds and compactness in (3) and some suitable infinite-dimen\-sio\-nal version of 
Schwarz's lemma to establish the desired contraction property of the underlying renormalization operator. 
\end{enumerate}

This strategy was put forth by Sullivan in \cite{Su} and it has motivated several breakthroughs in 
one-dimensional dynamics, especially in the study of real-analytic unimodal maps of the interval -- see 
the seminal works by McMullen \cite{McM}, Lyubich \cite{Ly1,Ly2}, and Avila-Lyubich \cite{AL} (for the case 
of $C^r$ unimodal maps, see also \cite{FMP}). 

Sullivan's strategy has been completely worked out also for critical circle 
maps having a unique critical point (of cubic type): see \cite{dF,dFdM,dFdM2,Gua,GdM,H,KT,KY,S,Yam1,Yam2,Yam,Y}. 
For maps with two or more critical points, however, much 
remains to be done. Only steps (1) and (2) have been established so far. Step (1) follows from unpublished work by 
Herman \cite{H} (based on previous work by Swiatek \cite{S}). In a recent breakthrough, Clark and van Strien \cite{CS}
establish step (2) in a very general context comprising multimodal maps of the interval and critical 
circle homeomorphisms with several (non-flat) critical points  -- henceforth called {\it multicritical circle maps\/}. 
Their proof is deep and rather involved, and it uses complex-analytic tools. 
They are primarily interested in the interval multimodal case, which is much more difficult than the 
multicritical circle case even at the topological or combinatorial level. 

Our goal in the present paper is to establish step (2) for multicritical circle maps 
using {\it purely real-variable techniques\/}. This main result can be stated as follows.

\begin{customthmA}{A}\label{maintheorem}
 Let $f,g:S^1\to S^1$ be two $C^r$ ($r\geq 3$) multicritical circle maps with the same irrational rotation number 
and the same number of (non-flat) critical points, and 
let  $h:S^1\to S^1$ be homeomorphism conjugating $f$ to $g$, {\it i.e.\/} such that 
$h\circ f=g\circ h$. If $h$ maps each critical point of $f$ to a corresponding critical point of $g$, 
then $h$ is quasisymmetric. 
\end{customthmA}

Unlike \cite{CS}, where all critical points are assumed to be of integral power-law type, 
here it is not assumed that the critical exponents are integers. Non-integral power laws 
are relevant, as they come about naturally {\it e.g.\/} in the study of certain one-dimensional maps arising as return maps 
to cross-sections of the Lorenz flow \cite{GW}. 

The proof of Theorem \ref{maintheorem} will be given in \S \ref{secmainthm}. Along the way, we will re-establish 
the real bounds, {\it i.e.\/} step (1), imitating the approach used in \cite{dFdM}. We stress that, in the case of circle maps 
having a {\it single\/} (non-flat) critical point, Theorem \ref{maintheorem} is due to Yoccoz 
(see the unplublished  manuscript \cite{Y2}, or \cite{dFdM} for a published account).   

We remark that the existence of a homeomorphism $h$ conjugating $f$ and $g$ as above is a well-known theorem also due 
to Yoccoz \cite{yoccoz}: indeed, every multicritical circle map without periodic points is topologically conjugate 
to an irrational rotation (in particular, such a multicritical circle map is uniquely ergodic). 
However, there is no reason why a conjugacy $h$ between $f$ and $g$ should map critical points 
of $f$ to critical points of $g$. A necessary and sufficient condition for this to happen can be 
stated as follows. Let $\mu_f$ and $\mu_g$ denote the unique Borel probability measures invariant under 
$f$ and $g$, respectively. Let $c_0(f),c_1(f),\ldots,c_{N-1}(f)\in S^1$ be the critical points of $f$, and
$c_0(g),c_1(g),\ldots,c_{N-1}(g)\in S^1$ be the critical points of $g$; both sets of critical points are 
assumed to be cyclically ordered (say counterclockwisely). 
Then a homeomorphism $h$ conjugating $f$ to $g$ and satisfying $h(c_i(f))=c_i(g)$ for all $0\leq i\leq  N-1$ 
exists if and only if $\mu_f[c_{i-1}(f),c_i(f)]=\mu_g[c_{i-1}(g),c_i(g)]$ for all for all $1\leq i\leq N-1$. 
The proof is straightforward. For more on the ergodic theory of multicritical circle maps, see \cite{dFG}. 

Still within the realm of circle maps, other interesting (partial) rigidity results have been obtained. 
For instance, in a recent paper \cite{Pa}, Palmisano proved that $C^2$ weakly order-preserving circle maps 
with a flat interval are quasi-symmetrically rigid in their non-wandering sets, provided their rotation number 
is of bounded type and a certain bounded geometry hypothesis is satisfied. 

\subsection{How the paper is organized} In \S \ref{prelim}, we introduce the basic concepts and well-known results 
to be used in the rest of the paper, including the cross-ratio distortion tools which are 
so ubiquitous in one-dimensional dynamics. In \S \ref{sec:realbounds}, we establish {\it real a-priori bounds\/} for 
multicritical circle maps. These bounds show that the dynamical partitions of a multicritical circle map have bounded 
geometry. The result, in slightly different form, is known from unpublished notes by Herman \cite{H}, based 
on previous work by Swiatek \cite{S}, and a nice exposition of Herman's approach 
is given by Petersen \cite{P} (see also the translation of Herman's notes by Ch\'eritat). 
Nevertheless, we provide a different proof, based on the one given for unicritical circle maps by 
de Faria and de Melo in \cite{dFdM}. The original contribution of the present paper starts in \S \ref{sec:geom}. 
There, we prove two crucial lemmas about the dynamical partitions of a multicritical circle map. 
The first lemma states that, if two atoms belonging to the level $n$ dynamical partitions associated to two distinct 
critical points intersect, then they must have comparable lenghts. The second states that, if an atom belonging to
a dynamical partition (at a certain level $n$) contains critical points of the appropriate (level $n$) return map, 
then it must be relatively large ({\it i.e.,} it must be comparable to the atom of the partition at the previous 
level that contains it). These two key facts are used to build a sequence of auxiliary partitions which are intermediate 
refinements of the dynamical partitions of a (chosen) critical point of the original map. In \S \ref{secmainthm}, 
this sequence of auxiliary partitions is used to build a {\it fine grid\/} (as originally defined in \cite{dFdM}). In the same 
section, we prove a criterion for a conjugacy between two given multicritical circle maps to be quasisymmetric: 
if such conjugacy yields an isomorphism between the fine grids of both maps, then it is quasisymmetric. 
From this criterion (stated but not proved in \cite{dFdM}) and the fine grid construction, our Theorem \ref{maintheorem} 
easily follows. Finally, in \S \ref{sec:final}, we state a conjecture concerning the smooth rigidity of 
multicritical circle maps which is naturally inspired by what happens in the unicritical case. 

\subsection{An important remark} In this paper, we require our circle maps to be at least $C^3$-smooth. The expert 
reader will certainly be familiar with the fact that the cross-ratio distortion tools we use, and consequently the real bounds we prove in 
\S \ref{sec:realbounds}, are valid in more generality. Indeed, for such purposes it suffices to assume that such a map 
$f$ is $C^1$ (with non-flat critical points) and that $\log{Df}$ satisfies a Zygmund condition (see \cite[Ch IV, \S 2]{dMvS}). 
However, in order to build the fine grid mentioned above, and control its geometry, we make fundamental use of the so-called 
{\it Yoccoz Lemma\/} (see Lemma \ref{lemyoccoz}). This lemma requires a negative Schwarzian condition, and therefore the $C^3$ hypothesis 
is, under current technology, unavoidable.

\section{Preliminaries}\label{prelim}

In this section we present the basic definitions and basic (well-known) results to be used throughout. 
For the facts presented here, plus general background on one-dimensional dynamics and much more, the reader should 
consult \cite{dMvS}. 

\subsection{Circle homeomorphisms} \label{circhom}
For us, the {\it unit circle\/} is the affine one-manifold $S^1=\mathbb{R}/\mathbb{Z}$. 
The dynamical systems we are interested in are {\it orientation-preserving homeomorphisms\/} of the unit circle. 
Let $f:S^1\to S^1$ be such a homeomorphism. It is well-known since Poincar\'e that the relative order of 
the points $f^n(x)$ ($n\in \mathbb{Z}$) making up the full-orbit of $x$ under $f$ on the circle is independent of 
the point $x$. If we count the average number of times that the finite piece of orbit $\{x,f(x),\ldots, f^n(x)\}$ 
winds around $S^1$ and let $n\to \infty$, we get a limiting number which is also independent of $x$, and is called 
the {\it rotation number\/} of $f$. The rotation number is a {\it topological invariant\/}, in the sense that 
any two topologically conjugate circle homeomorphisms always have the same rotation number (this is true even if 
the homeomorphisms are merely topologically semi-conjugate{\footnote{Given $f,g: S^1\to S^1$, we say that $f$ is {\it topologically semi-conjugate\/} 
to $g$ if there exists a continuous map $h:S^1\to S^1$ such that $h\circ f=g\circ h$.}}). 
As it happens, every orientation-preserving homeomorphism of the unit circle is topologically semi-conjugate to a rotation.

Given a homeomorphism $f$ as above, let $\rho$ be the rotation number of $f$, and consider its continued fraction expansion
\begin{equation*}
      \rho(f)= [a_{0} , a_{1} , \cdots ]=   
      \cfrac{1}{a_{0}+\cfrac{1}{a_{1}+\cfrac{1}{ \ddots} }} \ ,
    \end{equation*}
which is finite or infinite according to whether $\rho$ is rational or irrational, respectively. 
We are only interested in the irrational case, which corresponds to homeomorphisms without periodic orbits. 
Under this hypothesis, we may recursively define  an infinite sequence of 
{\it return times\/} associated to $\rho$ by
\begin{equation*}
 q_{0}=1, \hspace{0.4cm} q_{1}=a_{0}, \hspace{0.4cm} q_{n+1}=a_{n}q_{n}+q_{n-1} \hspace{0.3cm} \text{for $n \geq 1$} .
\end{equation*}
These numbers are precisely the denominators of the reduced fractions obtained by truncating the continued fraction expansion of $\rho$ 
at the $n-th$ level, that is to say
\begin{equation*}
 \frac{p_{n}}{q_{n}}= [a_{0} , a_{1} , \cdots , a_{n-1}]=   
      \cfrac{1}{a_{0}+\cfrac{1}{a_{1}+\cfrac{1}{  \ddots \cfrac{1}{a_{n-1} }} }}
\end{equation*}
These return times are characterized by the following property: For any $x \in S^{1}$, the closed interval $I_n(x)\subset S^1$ with endpoints $x$ and $f^{q_{n}}(x)$ 
containing the point $f^{q_{n+2}}(x)$ contains no other iterate $f^{j}(x)$ with $1 \leq j \leq q_{n}-1$. When $f$ is a rigid rotation (by an angle which is 
an irrational multiple of $2\pi$), the sequence $(q_n)$ is such that $|f^{q_{n+1}}(x)-x|< |f^{q_{n}}(x)-x|$ for all $n\geq 0$; for this reason, the 
return times $q_n$ are also called {\it closest return times\/}, and the points $f^{q_n}(x)$ are the {\it closest returns\/} of the orbit of $x\in S^1$ to 
$x$. These returns alternate around $x$, {\it i.e.,\/} we have either $f^{q_n}(x)<x<f^{q_{n+1}}(x)$ or $f^{q_{n+1}}(x)<x<f^{q_{n}}(x)$, in the natural order 
on $S^1$ induced from $\mathbb{R}$. 

The most basic combinatorial fact about a circle homeomorphism $f$ (topologically conjugate to an irrational rotation) to be used throughout is the following. 
For each $n\geq 0$ and each $x\in S^1$, the collection of intervals
\[
 \mathcal{P}_n(x)\;=\; \left\{f^i(I_n(x)):\;0\leq i\leq q_{n+1}-1\right\} \;\cup\; \left\{f^j(I_{n+1}(x)):\;0\leq j\leq q_{n}-1\right\} 
\]
is a {\it partition of the unit circle (modulo endpoints)\/}, called the {\it $n$-th dynamical partition\/} associated to the point $x$. 
The intervals of the form $f^i(I_n(x))$ are called {\it long\/}, whereas those of the form $f^j(I_{n+1}(x))$ are called {\it short\/}. 
For each $n$, the partition $\mathcal{P}_{n+1}(x)$ is a (non-strict) refinement of $\mathcal{P}_{n}(x)$.  

\subsection{Multicritical circle maps}\label{sec:multi}
In this paper we study homeomorphisms of the circle of a very special kind, namely {\it multicritical circle maps\/}. 
Here is the formal definition.

\begin{definition}\label{def:multicritic}
A multicritical circle map  is an orientation-preserving homeomorphism $f: S^{1}\rightarrow S^{1}$ of 
class $C^{r}$, $r\geq 3$, 
having finitely many critical 
points $c_{0}, \cdots, c_{N-1}$ satisfying the following.  
There exist neighbourhoods $W_{i}\subseteq S^{1}$ of each $c_{i}$ such that
\begin{enumerate}
 \item [(1)] The map $f$ has negative Schwarzian derivative on $W_{i} \setminus \{ c_{i}\}$.
 \item [(2)] There exist constants $0<\alpha_{i}< \beta_{i}$ and  $s_{i}>1$ such that for all $x \in W_{i}$
 \begin{equation*}
  \alpha_{i}|x-c_{i}|^{s_{i}-1} < f'(x) < \beta_{i} |x-c_{i}|^{s_{i}-1}.
 \end{equation*}
 \item [(3)] The variation of $\log Df$ on $S^{1} \setminus \cup_{i=0}^{N-1}W_{i}$ is bounded by $\rho >0.$
\end{enumerate}
\end{definition}

The following property holds around each critical point of a multicritical circle map. 

\begin{newcustomthm}{}
For all $x,y \in S^1$ with  $|x-c_{i}| \leq |y-c_{i}|$, we have
 \begin{equation*}
  \dfrac{|f(x)-f(c_{i})| }{|f(y)-f(c_{i})|}  \leq \gamma_{i} \left(\dfrac{|x-c_{i}|}{|y-c_{i}|} \right)^{s_{i}} ,
 \end{equation*}
where $\gamma_{i}>0$ is a constant depending only on $\alpha_{i}, \beta_{i}, \rho$. 
\end{newcustomthm}

The proof follows directly from property (2) in the above definition. We remark {\it en passant\/} that each $s_i$ 
is called the {\it criticality\/} or {\it power-law type\/} of the corresponding critical point $c_i$. 

As already mentioned in the introduction, a well-known theorem due to Yoccoz \cite{yoccoz} states that every multicritical circle map 
without periodic points is topologically conjugate to an irrational rotation. This basic fact will be used throughout. 

\subsection{Cross-ratios}
There are several types of cross-ratios used in 
one-dimensional dynamics. We describe here two of the most ubiquitous (but shall use only one of them). 

Let us denote by $N$ either the unit circle $S^1$ or the real line $\mathbb{R}$. Given two intervals 
$M\subset T\subset N$ with $M$ compactly contained in $T$, let us denote by $L$ and $R$ the two connected 
components of $T\setminus M$. We define the $a$-cross-ratio and the $b$ cross-ratio of the pair $(M,T)$, 
respectively, as follows:
\begin{equation*}
a(M,T)= \frac{|M| |T|}{|L| |R|}\ ,\ 
b(M,T)= \frac{|L| |R|}{|L\cup M| |M \cup R|} \ .
\end{equation*}
One easily checks that $b(M,T)^{-1}=1+ a(M,T)$. Both cross-ratios are preserved by Moebius transformations;  
 the latter is weakly contracted by maps with negative Schwarzian derivative (see below), whereas the former 
is weakly expanded. 

Unlike \cite{dFdM}, where the $a$-cross-ratio was used throughout, in the present paper it will be more 
convenient to use the $b$-cross-ratio. The latter has the advantage that its logarithm is 
given by the {\it Poincar\'e length\/} of $M$ inside $T$. More precisely,
\begin{equation}\label{bratio}
 \log{b(M,T)}\;=\; -\int_{M} \rho_T(x)\,dx \ ,
\end{equation}
where $\rho_T(x)$ is the {\it Poincar\'e density\/} of $T=[\alpha,\beta]$, given by
\begin{equation*}
 \rho_T(x)\;=\;\frac{\beta-\alpha}{(x-\alpha)(\beta-x)}\ .
\end{equation*}

\subsection{Distortion tools}\label{sec:distort}
The main tool used in this paper is {\it cross-ratio distortion\/}. 
Let $f:N\to N$ be a continuous map, and let $U\subseteq N$ be an open set such that $f|_{U}$ is a 
homeomorphism onto its image. If 
$M\subset T\subset U$ are intervals, with $M$ {\it compactly contained\/} in $T$ (written $M\Subset T$),  
the \textit{cross-ratio distortion} of the map $f$ on the pair of intervals $(M,T)$ is defined to be the ratio
\begin{equation*}
 D(f;M,T)= \frac{b(f(M),f(T))}{b(M,T)}.
\end{equation*}
If $f|_{T}$ is the restriction of a projective (Moebius) transformation, then one can easily see that $D(f;M,T)=1$. 
Also, when $f|_{T}$ is a diffeomorphism onto its image and $\log{Df}|_{T}$ has {\it bounded variation\/} in $T$, 
then an easy calculation using the mean value theorem shows that $D(f;M,T)\leq e^{2V}$, 
where $V=\mathrm{Var}(\log{Df}|_{T})$.   

Now, if $f|_{U}$ is a diffeomorphism onto its image, we define $\delta_f:U\times U\to \mathbb{R}$ by
\[
 \delta_f(x,y)\;=\;\left\{\begin{matrix} {\displaystyle{\frac{f(x)-f(y)}{x-y}}} & \ \ \textrm{if}\ \ x\neq y  \\
                                           {} & {} \\
                                          {\displaystyle{f'(x)}}  &   \ \ \textrm{if}\ \ x= y 
                           \end{matrix}\right.
\]
If $f$ is $C^3$ then $\delta_f$ is $C^2$, and the following facts are straightforward.{\footnote{The mixed partial derivative 
appearing in \eqref{logdist} is, up to a multiplicative constant, what one calls the {\it bi-Schwarzian\/} 
of $f$. More precisely, the bi-Schwarzian $B_f$ is defined as 
\[
 B_f(x,y)\;=\; 6\,\frac{\partial^2\delta_f}{\partial x\partial y}(x,y)\ .
\]
Clearly, $B_f(x,y)\to Sf(x)$ as $y\to x$, hence the name. The bi-Schwarzian is a cocycle, in the sense that it satifies a chain rule: 
If $f,g$ are $C^3$ maps for which $f\circ g$ makes sense, then $B_{f\circ g}(x,y)\;=\;g'(x)g'(y)B_f(g(x),g(y)) + B_g(x,y)$. 
Thus is, of course, entirely consistent with the chain rule for the Schwarzian. We will make no further use of the bi-Schwarzian in this paper.}}
\begin{enumerate}
 \item[(i)] For all $M\subset T\subset U$,
\begin{equation}\label{logdist}
 \log{D(f;M,T)}\;=\; \doublint_{M\times T} \frac{\partial^2\delta_f}{\partial x\partial y}\,dxdy\ .
\end{equation}
 \item[(ii)] For all $x\in U$ we have
\[
 \lim_{y\to x} \frac{\partial^2\delta_f}{\partial x\partial y}(x,y)\;=\; \frac{1}{6}Sf(x)\ ,
\]
where 
\[
 Sf\;=\; \left(\frac{f''}{f'}\right)'-\frac{1}{2}\left(\frac{f''}{f'}\right)^{2}
\]
is the Schwarzian derivative of $f$. 
\end{enumerate}

These two facts put together yield the following.

\begin{lemma}
 If $f$ is $C^3$ and $Sf<0$, then for all sufficiently small intervals $M\subset T$ we have 
$D(f;M,T)<1$. \qed
\end{lemma}

In other words, a map with negative Schwarzian derivative contracts (small) cross-ratios. 


Now we have the following fundamental result. Given a family of intervals $\mathcal{F}$ on $N$ and a positive integer $m$, 
we say that $\mathcal{F}$ has {\it multiplicity of intersection at most $m$\/} if each $x\in N$ belongs to at most $m$ 
elements of $\mathcal{F}$. 

\begin{customthm}{} Given a multicritical critical circle map $f:S^1\to S^1$, there exists a constant $C>1$, depending only on $f$, 
such that the following holds. If $M_i\Subset T_{i} \subset S^1$, where $i$ runs through some finite set of indices $\mathcal{I}$, 
are intervals on the circle such that the family $\{T_i: i\in \mathcal{I}\}$ 
has multiplicity of intersection at most $m$, then 
 \begin{equation}\label{crossprod}
  \prod_{i \in \mathcal{I}} D(f;M_{i},T_{i}) \leq C^{m}\ .
 \end{equation}
\end{customthm}

\begin{proof}[Sketch of proof] 
Let $\mathcal{U}=\bigcup W_i$, where the $W_i$'s are as in Definition \ref{def:multicritic}, and let $\mathcal{V}$ 
be an open set with $\mathcal{U}\cup \mathcal{V}=S^1$ whose closure does not contain any critical point of $f$. 
We assume without loss of generality that the maximum length of the $T_i$'s is smaller than the Lebesgue number of the covering 
$\{\mathcal{U},\mathcal{V}\}$. 
Write the product on the left-hand side of \eqref{crossprod} as $P_1\cdot P_2$, where 
\[
 P_1\;=\;\prod_{T_i\subseteq \mathcal{V}}  D(f;M_{i},T_{i})\ \ \ ,\ \ \ P_2\;=\;\prod_{T_i\subseteq \mathcal{U}}  D(f;M_{i},T_{i})\ .
\]
Then on the one hand $P_1\leq e^{2mV}$, where $V=\mathrm{Var}(\log{Df}|_{\mathcal{V}})$. On the other hand, 
the factors making up $P_2$ are of two types: those such that $f|_{T_i}$ is a diffeomorphism onto its image, and those 
such that $T_i$ contains some critical point of $f$. All factors of the first type are diffeomorphisms with negative Schwarzian 
and therefore satisfy $D(f;M_i,T_i)<1$. Factors of the second type are controlled by the Power Law 
(or equivalently by property (2) in Definition \ref{def:multicritic}), and since there are at most $mN$ such factors 
(where $N$ is the number of critical points of $f$), the result follows. 
For more details, see \cite{S}.
\end{proof}

Finally, there will be only two moments in this paper, namely in the proofs of Lemma \ref{negSret} and Corollary \ref{balancedbridges}, where 
we will need the so-called {\it Koebe distortion principle\/}, a well-known tool for controlling non-linearity. Here is the statement.

\begin{lemma}[Koebe distortion principle] \label{koebe}
For each $\ell,\tau>0$ and each multicritical circle map $f$ there exists a 
constant $K=K(\ell,\tau,f)>1$ with the following property. If
$T$ is an interval such that $f^k|_{T}$ is a diffeomorphism onto its image and if $\sum_{j=0}^k |f^j(T)|\leq \ell$, 
then for each interval $M\subset T$ for which $f^k(T)$ contains a $\tau$-scaled neighborhood
of $f^k(M)$ one has
\[
\frac{1}{K}\leq \frac{|Df^k(x)|}{|Df^k(y)|}\leq K
\]
for all $x,y\in M$. 
\end{lemma}

The proof of this lemma can be found in \cite[p.~295]{dMvS}. 

\subsection{Quasisymmetry}\label{sec:quasisym}
As we stated in the introduction, our goal in the present paper is to show that a topological conjugacy 
between two multicritical circle maps has a geometric property known as {\it quasisymmetry\/} (provided it 
maps the critical points of one map to the critical points of the other). An orientation-preserving 
homeomorphism $h: S^1\to S^1$ is said to be {\it quasisymmetric\/} if there exists a constant $K\geq 1$ such that, 
for all $x\in S^1$ and all $t>0$, we have 
\[
 \frac{1}{K}\;\leq\; \frac{|h(x+t)-h(x)|}{|h(x)-h(x-t)|}\;\leq\; K\ .
\]
Quasisymmetric homeomorphisms of the circle are precisely the boundary values of {\it quasiconformal\/} homeomorphisms of the unit disk 
in the complex plane (see \cite{Ah}). As such, they can be very bad from the differentiable viewpoint. In fact, they are often 
purely singular with respect to Lebesgue measure on the circle.

\section{The real a-priori bounds theorem}\label{sec:realbounds}

 In this section we establish real a-priori bounds for multicritical circle maps, {\it i.e.\/} maps satisfying the hypotheses of  
Definition \ref{def:multicritic}. Let $f$ be such a map, and $c_0,c_1,\ldots, c_{N-1}$ be its critical points; we assume throughout that $f$ 
has no periodic points. For each critical point 
  $c_{k}$ with $0\leq k \leq N-1$ and each non-negative integer $n$, let $I_{n}(c_{k})$ be the interval with endpoints 
$c_{k}$ and $f^{q_{n}}(c_{k})$ containing 
  $f^{q_{n+2}}(c_{k})$, as in \S \ref{circhom}. We will often write $I_{n}^{j}(c_{k})=f^{j}(I_{n}(c_{k}))$ for all $j$ and $n$. 
Recall from \S \ref{circhom} that the \textit{$n$-th dynamical partition} of $f$ associated with the critical point $c_{k}$, namely  
$\mathcal{P}_{n}(c_{k})$, is given by
  \begin{equation*}
  \mathcal{P}_{n}(c_{k}) = \left\{ I_{n}^i(c_{k})\,:\, 0\leq i\leq q_{n+1}-1\right\} \;\cup\; \left\{I_{n+1}^j(c_{k})\,:\, 
0\leq j\leq q_{n}-1\right\} \ .
  \end{equation*}
  
Let us focus our attention, for the time being, on one of the critical points only, say $c_0$. Everything we will say below about $c_0$ and 
its associated dynamical partitions $\mathcal{P}_{n}(c_{0})$, can be said about any other critical point of $f$ and its associated partitions. 
To simplify the notation a bit, we shall write below $\mathcal{P}_{n}$ instead of $\mathcal{P}_{n}(c_{0})$; accordingly, the atoms of $\mathcal{P}_{n}$ 
will be denoted by $I_{n}^i, I_{n+1}^j$ instead of $I_{n}^i(c_0), I_{n+1}^j(c_0)$, respectively. 
For $n$ large enough, we may assume that no two critical points of $f$ are in the same atom of $\mathcal{P}_{n}$.

 \begin{theorem}[Real A-priori Bounds] \label{realbounds} 
 Let $f$ be a multicritical circle map. There exists a constant $C>1$ depending only of $f$ such that the following holds. 
For all $n\geq 0$ and for each pair of adjacent atoms $I, J\in \mathcal{P}_{n}$ we have
 \begin{equation}\label{compatom}
  C^{-1} {|J|} \leq |I| \leq C|J|.
 \end{equation}
 \end{theorem}

The inequalities in \eqref{compatom} tell us that the atoms $I$ and $J$ are {\it comparable\/}. Thus the above theorem is saying that 
{\it any two adjacent atoms of a dynamical partition of $f$ are comparable\/}. 

The proof of Theorem \ref{realbounds} is a bit long and depends on a few auxiliary lemmas. Rather then following the original 
unpublished notes of Herman \cite{H} 
(which in turn were based on previous work by Swiatek \cite{S}), we will imitate the approach used in \cite[\S 3]{dFdM}, 
There is one crucial difference, however. In that paper, the map $f$ had only one critical point (our $c_0$ here), and therefore all 
transition maps $f^{i_2-i_1}: I_n^{i_1}\to I_n^{i_2}$ with 
$1\leq i_1<i_2\leq q_{n+1}$ were {\it diffeomorphisms\/}. Hence the authors were able to use the {\it Koebe distortion principle\/} 
(see Lemma \ref{koebe} in \S \ref{sec:distort}). We are not allowed to do that here, due to the presence of other 
critical points (besides $c_ 0$). 
Instead, we deal directly with the control of cross-ratio distortion, via the Cross-Ratio Inequality. 

\subsection{Symmetric intervals are comparable}
First we establish a comparability result such as inequality \eqref{compatom} for general dynamically symmetric intervals, 
that is, any pair of intervals with an endpoint in common $x\in S^1$, the other 
 endpoints being $f^{q_{n}}(x)$ and $f^{-q_{n}}(x)$, for some $n>0$. For this purpose
 we need the next two lemmas. The first lemma is proved by what deserves to be called a {\it seven point argument\/} (even though only 
five points appear in the statement, seven points are used in the proof). 
 
\begin{lemma}\label{lemma1}
There exists a constant $C_1>1$ depending only on $f$ satisfying the following. 
For each $n\geq 0$ there exist $z_{1}, z_{2}, z_{3}, z_{4}$ and $z_{5}$ points in $S^{1}$ with $z_{j+1}=f^{q_{n}}(z_{j})$ such that
 \begin{equation}\label{manyzees}
  C_1^{-1} \;\leq\;  \dfrac{|z_{i-1}-z_{i}|}{|z_{i+1}-z_{i}|}  \;\leq\; C_1 , \hspace{0.4cm} \text{for $i = 2,3,4 $}.
 \end{equation}
\end{lemma}

\begin{proof}
Let $z \in S^{1}$ be a point such that, for all $x \in S^{1}$,
 \begin{equation*}
  |f^{q_{n}}(z)-z| \leq |f^{q_{n}}(x)-x|.
 \end{equation*}
Then consider the seven points
\[
 z_{0}=f^{-4q_{n}}(z)\,,\, z_{1}=f^{-3q_{n}}(z)\,,\, z_{2}=f^{-2q_{n}}(z)\,,\, z_{3}=f^{-q_{n}}(z)\,,\, 
\]
\[ 
 z_{4}=z\,,\, z_{5}=f^{q_{n}}(z)\,,\,z_{6}=f^{2q_{n}}(z)\ . 
\]
Note that, by our choice of $z$,   
\begin{equation}\label{z4z5}
 |z_{4}-z_{5}| \leq  |z_{i}-z_{i+1}| \ ,\ \ \textrm{for all}\  0 \leq i \leq 5 \ .
\end{equation}
These seven points are cyclically ordered as given (either in clockwise or counterclockwise order in the circle), provided $n$ is sufficiently large. 
Let $J\subset S^1$ be the closed interval with endpoints $z_0$ and $z_6$ that contains $z=z_4$. 
For each $0\leq i\leq 3$, let $T_i=[z_i,z_{i+3}]\subset J$ and $M_i=[z_{i+1},z_{i+2}]\subset T_i$. Then the homeomorphism $f^{q_n}$ maps $T_i$ onto 
$T_{i+1}$ and $M_i$ onto $M_{i+1}$, for $0\leq i\leq 2$. Moreover, the collection of intervals $\{T_i\,,\,f(T_i)\,,\,\ \ldots\,,\,f^{q_n}(T_i)\}$ has 
intersection multiplicity equal to $3$. 
\begin{enumerate}
\item[(i)] Let us first prove \eqref{manyzees} for $i=4$. Applying the Cross-Ratio Inequality to $f^{q_n}$ and the pair $(M_2,T_2)$, we have 
\[
 D(f^{q_n};M_2,T_2)\;=\; \frac{b(M_3,T_3)}{b(M_2,T_2)}\;=\; \dfrac{|z_{3}-z_{4}||z_{5}-z_{6}||z_{2}-z_{4}|}{|z_{4}-z_{6}||z_{2}-z_{3}||z_{4}-z_{5}|} 
\leq B \ ,
\]
where $B>1$ is a constant that depends only on $f$. 
But then, using \eqref{z4z5}, we see that 
\begin{equation*}
 \dfrac{|z_{3}-z_{4}|}{|z_{4}-z_{5}|} \leq B \dfrac{|z_{4}-z_{6}|}{|z_{5}-z_{6}|} 
 = B \left( \dfrac{|z_{4}-z_{5}|}{|z_{5}-z_{6}|} + 1 \right) \leq 2B.
\end{equation*}
Therefore, defining $B_1=2B$ and again using \eqref{z4z5}, we get
\begin{equation}\label{ineq2.1}
 {B_1^{-1}} \leq \dfrac{|z_{3}-z_{4}|}{|z_{4}-z_{5}|} \leq B_1\ .
\end{equation}

\item[(ii)] Let us now prove \eqref{manyzees} for $i=3$.
Applying the Cross-Ratio Inequality to $f^{q_n}$ and the pair $(M_1,T_1)$, we have 

\begin{equation*}
 D(f^{q_n};M_1,T_1)\;=\; \frac{b(M_2,T_2)}{b(M_1,T_1)}\;=\; \dfrac{|z_{2}-z_{3}||z_{4}-z_{5}||z_{1}-z_{3}|}{|z_{3}-z_{5}||z_{1}-z_{2}||z_{3}-z_{4}|} 
\leq B\ ,
\end{equation*}
or equivalently, using \eqref{z4z5} and the upper bound in \eqref{ineq2.1},
\begin{equation*}
  \dfrac{|z_{2}-z_{3}|}{|z_{3}-z_{4}|} \leq B \dfrac{|z_{3}-z_{5}|}{|z_{4}-z_{5}|} 
  \leq B \left( \dfrac{|z_{3}-z_{4}|}{|z_{4}-z_{5}|} + 1 \right) \leq B(B_{1}+1).
\end{equation*}
On the other hand, using \eqref{z4z5} once again,
\begin{equation*}
 \dfrac{|z_{3}-z_{4}|}{|z_{2}-z_{3}|} \leq \dfrac{|z_{3}-z_{4}|}{|z_{4}-z_{5}|} \leq B_1.
\end{equation*}
Taking $B_{2} = B(B_{1}+1)$ and putting the last two inequalities together, we get
\begin{equation}\label{ineq2.2}
{B_2^{-1}} \leq \dfrac{|z_{2}-z_{3}|}{|z_{3}-z_{4}|} \leq B_2\ .
\end{equation}
\item[(iii)] Finally, let us prove \eqref{manyzees} for $i=2$. As before, applying the Cross-Ratio inequality to $f^{q_n}$ and the pair $(M_0,T_0)$, we have 

\begin{equation*}
 D(f^{q_n};M_0,T_0)\;=\; \frac{b(M_1,T_1)}{b(M_0,T_0)}\;=\; \dfrac{|z_{1}-z_{2}||z_{3}-z_{4}||z_{0}-z_{2}|}{|z_{2}-z_{3}||z_{0}-z_{1}||z_{2}-z_{3}|} 
\leq B\ ,
\end{equation*}
From this, using \eqref{z4z5} and \eqref{ineq2.2}, we get on the one hand
\begin{equation}\label{ineq2.3.1}
 \dfrac{|z_{1}-z_{2}|}{|z_{2}-z_{3}|} \leq B \dfrac{|z_{2}-z_{4}|}{|z_{3}-z_{4}|} 
 \leq B \left( \dfrac{|z_{2}-z_{3}|}{|z_{3}-z_{4}|} +1 \right) \leq B(B_{2}+1).
\end{equation}
On the other hand, the inequalities (\ref{ineq2.1}) and (\ref{ineq2.2}) tell us that
\begin{equation}\label{ineq2.3.2}
 \dfrac{|z_{2}-z_{3}|}{|z_{1}-z_{2}|} \leq B_{2} \dfrac{|z_{3}-z_{4}|}{|z_{1}-z_{2}|} \leq B_{2}B_{1}\dfrac{|z_{4}-z_{5}|}{|z_{1}-z_{2}|}\leq B_{2}B_{1}\ .
\end{equation}
Defining $B_{3}= \max \{ B(B_{2}+1), B_{2}B_{1} \}=B_1B_2$, and using inequalities (\ref{ineq2.3.1}) and (\ref{ineq2.3.2}), we obtain
\begin{equation}\label{ineq2.3}
 B_{3}^{-1} \leq \dfrac{|z_{1}-z_{2}|}{|z_{2}-z_{3}|} \leq B_{3} \ .
\end{equation} 
\end{enumerate}
Summarizing, we have proved \eqref{manyzees} with $C_1=\max\{B_1,B_2,B_3\}=B_3>1$, a constant that indeed depends only on $f$.  
\end{proof}

\begin{lemma}\label{lemma2}
 There exists a constant $C_{2}>1$ depending only on $f$ satisfying the following. 
 Let $z_{1},z_{2},z_{3},z_{4}$ and $z_{5}$ be the points given by Lemma \ref{lemma1}. If $w_{0},w_{1},w_{2},w_{3}$ and $w_{4}$ are points on 
 the circle such that $w_{j+1}=f^{q_{n}}(w_{j})$ and such that $w_{1}$ lies in the interval with endpoints $z_{1}$ and $z_{2}$ that does not contain
 $z_{3}$, then  
 \begin{equation}\label{manyws}
   \dfrac{|w_{1}-w_{2}|}{|w_{0}-w_{1}|}\leq {C_{2}}\hspace{1.0cm} \text{ and} \hspace{1.0cm} 
   C_2^{-1}\leq \dfrac{|w_{i-1}-w_{i}|}{|w_{i}-w_{i+1}|}\leq {C_{2}} \ \ \text{for } i=2,3\ .
 \end{equation}
\end{lemma}

\begin{proof}
To prove the first inequality, we consider the interval $T$ with endpoints $w_{0}$ and $w_{3}$ containing $z_{1}, w_{1}, z_{2}, w_2, z_3$, and 
the subinterval $M=[w_1,w_2]\subset T$. Note that $\{T,f(T),\ldots,f^{q_n}(T)\}$ has intersection multiplicity equal to $3$. 
Hence, applying the Cross-Ratio Inequality to $f^{q_{n}}$ and the pair $(M,T)$, 
we get $b(f^{q_n}(M),f^{q_n}(T))\leq B b(M,T)$, or equivalently 
 \begin{equation}\label{ineq2.4}
  \dfrac{|w_{1}-w_{2}| |w_{3}-w_{4}|}{|w_{1}-w_{3}| |w_{2}-w_{4}|} \leq B \dfrac{|w_{0}-w_{1}| |w_{2}-w_{3}|}{|w_{0}-w_{2}| |w_{1}-w_{3}|} \ . 
 \end{equation}
Since the points $w_0,z_1,w_1,\ldots,z_4,w_4,z_5$ are cyclically ordered as given, we have the inequalities 
$|z_{1}-z_{2}|\leq |w_{0}-w_{2}|$, $|w_{2}-w_{3}| \leq |z_{2}-z_{4}|$, and $|w_{2}-w_{4}|\leq |z_{2}-z_{5}|$. Moreover, we have $|z_{4}-z_{5}|\leq |w_{3}-w_{4}|$, 
by our choice of $z=z_4$ in Lemma \ref{lemma1}. These facts, when put back into \eqref{ineq2.4}, yield
 \begin{equation}\label{firstw}
  \dfrac{|w_{1}-w_{2}|}{|w_{0}-w_{1}|} \leq B \dfrac{|z_{2}-z_{4}| |z_{2}-z_{5}|}{|z_{1}-z_{2}| |z_{4}-z_{5}|}\leq B(C_1+C_1^2)(1+C_1+C_1^2) \ ,
 \end{equation}
where we have used the inequalities of Lemma \ref{lemma1}. 

To prove the upper bound in the last two inequalities in \eqref{manyws}, we simply 
note that $|w_i-w_{i+1}|\geq |z_4-z_5|$ and that $|w_{i-1}-w_i|\leq |z_{i-1}-z_{i+1}|$. Using the 
inequalities \eqref{manyzees}, we deduce that
\begin{equation}\label{secondw}
 \dfrac{|w_{i-1}-w_i|}{|w_{i}-w_{i+1}|}\leq \dfrac{|z_{i-1}-z_i|}{|z_4-z_5|}+ \dfrac{|z_{i}-z_{i+1}|}{|z_4-z_5|}\leq 2C_1^3
\end{equation}
The lower bound for the same inequalities in  \eqref{manyws} is proven in exactly the same way (the value obtained is $(2C_1^3)^{-1}$). 
Thus, \eqref{manyws} is established, provided we take $C_{2}= \max\{ 2C_1^{3}\,,\, B(C_1+C_1^2)(1+C_1+C_1^2)\}$.
\end{proof}

We are now in a position to show that dynamically symmetric intervals are always comparable. This fact will be crucial in the proof of 
Proposition \ref{lemma4}, which in turn will be the major step in the proof of Theorem \ref{realbounds}. In the lemma below, we make use of the 
following simple remark. Given $\xi\in S^1$, let $J_n(\xi)\subset S^1$ be the interval with endpoints $f^{-q_n}(\xi)$ and $f^{q_n}(\xi)$ that contains 
$\xi$. Then $\bigcup_{i=0}^{q_{n+1}} f^{-i}(J_n(\xi)) = S^1$. 

\begin{lemma}\label{lemma3}
 There exists a constant  $C_3> 1$ depending only on $f$ such that, for all $n\geq 0$ and all $x \in S^{1}$, we have
 \begin{equation}\label{symreturns}
  C_3^{-1}{|x- f^{-q_{n}}(x)|} \leq |f^{q_{n}}(x)-x| \leq C_3 |x- f^{-q_{n}}(x)|.
 \end{equation}
\end{lemma}

\begin{proof}
Note that it suffices to prove the second of the two inequalities in \eqref{symreturns} for all $x$
(to get the first inequality from the second, just replace $x$ by $f^{-q_n}(x)$). 

Thus, let $x \in S^{1}$ and let $0\leq i \leq q_{n+1}$ such that $f^{i}(x)$ lies on the interval $J$ with endpoints $z_{1}$ and $z_{3}$ that contains $z_2$, 
where $z_{1}, z_{2}, \cdots, z_{5}$ are the points given by Lemma \ref{lemma1}. Such an $i$ exists because of the simple remark preceding 
the present lemma, applied to $\xi=z_2$ (so that $J_n(z_2)=J$). Then either $f^i(x)\in [z_1,z_2]\subset J$, or $f^i(x)\in (z_2,z_3]\subset J$. 
We prove the lemma assuming the former case (the proof 
in the latter case being similar). 

Let us consider the points $w_{0}=f^{i-q_{n}}(x), w_{1}=f^{i}(x), w_{2}= f^{i+q_{n}}(x)$ and 
 $w_{3}=f^{i+2q_{n}}(x)$. Then we are in the situation of Lemma \ref{lemma2}. 
Consider the interval $T$ with endpoints $f^{-q_{n}}(x)$ and $f^{2q_{n}}(x)$ that contains
 $x$, and let $M=[x,f^{q_n}(x)]\subset T$. Note that 
\begin{equation}\label{ineq1lem3}
 b(M,T)=\dfrac{|x-f^{-q_{n}}(x)| |f^{q_{n}}(x)-f^{2q_{n}}(x)|}{|f^{q_{n}}(x)-f^{-q_{n}}(x)| |x-f^{2q_{n}}(x)|} 
 \leq \dfrac{|x- f^{-q_{n}}(x)|}{|f^{q_{n}}(x)-x|} \ .
\end{equation}
From the inequalities \eqref{manyws} in Lemma \ref{lemma2}, we also have
\begin{equation}\label{ineq2lem3}
 b(f^{i}(M),f^{i}(T))= \dfrac{|w_0-w_1| |w_2-w_3|}{|w_0-w_2||w_1-w_3|} \geq \dfrac{1}{(1+ C_2)^2}\ .
\end{equation}
Since $\{T,f(T),\ldots,f^{i}(T)\}$ has intersection multiplicity at most equal to $3$,  
the Cross-Ratio Inequality tells us that $b(f^{i}(M),f^{i}(T))\leq B b(M,T)$, where the constant $B$ is the same as in the previous lemmas. 
Combining this fact with \eqref{ineq1lem3} and \eqref{ineq2lem3}, we deduce that
\begin{equation}\label{ineq2.5}
  |f^{q_{n}}(x)-x| \leq  B(1+ C_2)^{2}|x-f^{-q_{n}}(x)|.
 \end{equation}
This proves \eqref{symreturns}, provided we take $C_3=B(1+ C_2)^2$.
\end{proof}

\subsection{Comparability of closest returns and beyond}

We now come to the major step towards the proof of Theorem \ref{realbounds}, namely Proposition \ref{lemma4} below.
Roughly speaking, it states that the atoms of the partition $\mathcal{P}_n(c_0)$ that are closest to the critical point $c_0$, including the 
closest return intervals $I_n(c_0)$ and $I_{n+1}(c_0)$, are pairwise comparable.

\vspace{0.5cm}

\begin{figure}[h]
\begin{center}
\psfrag{A}[][]{ $I_{n}^{q_{n+1}-q_{n}}$} 
\psfrag{B}[][][1]{$I_{n+1}$}
\psfrag{C}[][][1]{$I_{n}$} 
\psfrag{D}[][][1]{$I_{n}^{q_{n}}$}
\psfrag{E}[][][1]{$I_{n}^{q_{n+1}}$}
\psfrag{F}[][][1]{$I_{n+1}^{q_{n}}$}
\psfrag{c}[][]{$c_{0}$} 
\includegraphics[width=3.5in]{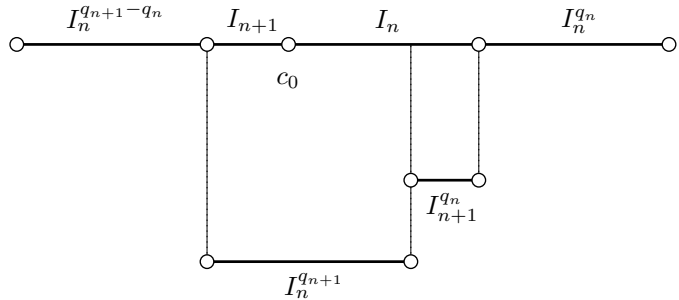}
\end{center}
\caption[slope]{\label{compare6} The six intervals of Proposition \ref{lemma4}.}
\end{figure}

\bigskip 

\begin{proposition}\label{lemma4}
 The six intervals in Figure \ref{compare6} are pairwise comparable. More precisely, there exists a constant $C_4>1$ 
depending only on $f$ such that, for all $n\geq 1$ and for all 
$I,J\in \{I_n,I_{n+1},I_n^{q_n}, I_n^{q_{n+1}}, I_{n+1}^{q_n},I_{n}^{q_{n+1}-q_{n}}\}$, we have 
\begin{equation}\label{labellem4}
 C_4^{-1} \leq \dfrac{|I|}{|J|}\leq C_4  \ .
\end{equation}
\end{proposition}

\begin{proof}
 We break up the proof into several steps, as follows. 
 \begin{enumerate}
  \item [(i)] {\it The intervals $I_{n}$ and $I_{n}^{q_{n}}$ are comparable\/}. Indeed, these two intervals are dynamically symmetric with respect to their 
common endpoint $f^{q_n}(c_0)$. Hence, by Lemma \ref{lemma3} we have 
\begin{equation}\label{ineqa1}
   C_3^{-1} |I_{n}| \;\leq\; |I_{n}^{q_{n}}| \;\leq\; C_3 |I_{n}| 
  \end{equation}
 \item[(ii)] {\it The intervals $I_{n}^{q_{n+1}}$ and $I_{n}^{q_{n+1}-q_{n}}$ are comparable\/}. 
Indeed, these two intervals are dynamically symmetric with respect to their common endpoint 
$f^{q_{n+1}}(c_0)$. Hence, again by Lemma \ref{lemma3} we have
  \begin{equation}\label{ineqa2}
   C_3^{-1}|I_{n}^{q_{n+1}}| \leq |I_{n}^{q_{n+1}-q_{n}}| \leq C_3|I_{n}^{q_{n+1}}|.
  \end{equation}
  \item [(iii)] {\it The intervals  $I_{n}^{q_{n+1}-q_{n}}$ and $I_n$ are comparable\/}. 
Consider the interval $I_{n}^{-q_{n}}$, with endpoints $c_{0}$ and $f^{-q_{n}}(c_{0})$. Since such interval is dynamically symmetric 
to the interval $I_n$, we have by the Lemma \ref{lemma3}
  \begin{equation}\label{ineqb1}
   C_3^{-1}|I_{n}^{-q_{n}}| \leq |I_{n}| \leq C_3|I_{n}^{-q_{n}}|.
  \end{equation}
 From the the right-hand side of inequality 
(\ref{ineqb1}), the inclusion $I_{n}^{-q_{n}}\subseteq I_{n}^{q_{n+1}-q_{n}}\cup I_{n}^{q_{n+1}}$ and the left-hand side of inequality (\ref{ineqa2}), we deduce that
  \begin{equation}\label{ineqb2}
   |I_{n}| \leq C_3(C_3+1)|I_{n}^{q_{n+1}-q_{n}}|.
  \end{equation}
  Now, we have $I_{n}^{q_{n+1}}\subseteq I_{n+1}\cup I_{n}$, and also $|I_{n+1}| \leq C_3|I_{n+1}^{-q_{n+1}}|$, because the intervals $I_{n+1}$ and 
and $I_{n+1}^{-q_{n+1}}$ are dynamically symmetric. Moreover, we have 
the inclusion $I_{n+1}^{-q_{n+1}}\subseteq I_{n}$. Combining these facts with the right-hand side of (\ref{ineqa2}), we get
  \begin{equation*}
   |I_{n}^{q_{n+1}-q_{n}}| \leq C_3(C_3+1)|I_{n}| \ .
  \end{equation*}
From this and (\ref{ineqb2}), we arrive at
  \begin{equation}\label{ineqb3}
   C_3^{-1}(C_3+1)^{-1}|I_{n}| \leq |I_{n}^{q_{n+1}-q_{n}}| \leq C_3(C_3+1)|I_{n}|. 
  \end{equation}
  \item [(iv)] {\it The intervals  $I_{n}$ and $I_{n+1}$ are comparable\/}. It is here that we use the power-law at the critical point $c_0$ in an essential way. 
First note that $I_{n+1}^{-q_{n+1}}\subseteq I_{n}$ and that the intervals $I_{n+1}^{-q_{n+1}}$ and $I_{n+1}$ are dynamically symmetric with respect to their common 
endpoint $c_0$. Hence, using Lemma \ref{lemma3} we get 
  \begin{equation}\label{ineqc1}
   |I_{n+1}| \leq C_3|I_{n}| \ .
  \end{equation}
The real issue here, thus, is to prove an inequality in the opposite direction. Let us consider the interval $T=I_{n+1}\cup I_n\cup I_n^{q_n}$ and its image 
$f(T)$ under $f$, which contains the critical value $f(c_0)$; note that the family 
$\{T,f(T),\ldots,f^{q_{n+1}}(T)\}$ has intersection multiplicity equal to 3. We look at the cross-ratio distortion of $f^{q_{n+1}-1}$ on the pair 
$(I_{n}^1, f(T))$. By the Cross-Ratio Inequality, we have
\begin{equation}\label{crithelps1}
 D(f^{q_{n+1}-1}; I_{n}^1, f(T))\;=\;\dfrac{b(I_n^{q_{n+1}}, f^{q_{n+1}}(T))}{b(I_n^1,f(T))}\;\leq\; B \ .
\end{equation}
But
\begin{equation}\label{crithelps2}
 b(I_n^{q_{n+1}}, f^{q_{n+1}}(T)) = \dfrac{|I_{n+1}^{q_{n+1}}|}{|I_{n+1}^{q_{n+1}}| + |I_n^{q_{n+1}}|} 
\cdot \dfrac{|I_n^{q_{n+1}+q_n}|}{|I_n^{q_{n+1}}| + |I_n^{q_{n+1}+q_n}|} \ .
\end{equation}
Since the intervals $I_n^{q_{n+1}+q_n}$ and $I_n^{q_{n+1}}$ are dynamically symmetric with respect to their common endpoint, we see 
from Lemma \ref{lemma3} that the second fraction on the right-hand side of \eqref{crithelps2} is bounded from below by $C_3^{-1}/(1+C_3)$. 
The intervals $I_{n+1}^{q_{n+1}}$ and $I_{n+1}$ are also dynamically symmetric with respect to their common endpoint, so again by Lemma \ref{lemma3}
we have $C_3^{-1}|I_{n+1}|\leq |I_{n+1}^{q_{n+1}}| \leq C_3|I_{n+1}|$; in addition, $I_{n+1}^{q_{n+1}}\subset I_{n+1}\cup I_n$, so that 
$|I_{n+1}^{q_{n+1}}|\leq |I_{n+1}| + |I_n|$. Putting all these facts back into \eqref{crithelps2}, we deduce that
\begin{equation}\label{crithelps3}
 b(I_n^{q_{n+1}}, f^{q_{n+1}}(T)) \geq \theta_1 \dfrac{|I_{n+1}|}{|I_n|} \ ,
\end{equation}
where $\theta_1=C_3^{-2}(1+C_3)^{-1}(1+C_3+C_3^2+C_3^3)^{-1}$. This bounds the numerator of \eqref{crithelps1} from below, 
so we proceed to bound the denominator from above.  We have
\begin{equation}
 b(I_n^1,f(T)) = \dfrac{|I_{n+1}^1|}{|I_{n+1}^1|+|I_n^1|}\cdot \dfrac{|I_n^{1+q_n}|}{|I_n^1|+|I_n^{1+q_n}|}\ .
\end{equation}
Since the intervals $I_n^1$ and $I_n^{1+q_n}$ are also dynamically symmetric with respect to their common endpoint, applying 
Lemma \ref{lemma3} yet again yields
\begin{equation}\label{crithelps4}
  b(I_n^1,f(T)) \leq \dfrac{C_3}{1+C_3} \dfrac{|I_{n+1}^1|}{|I_n^1|} \ .
\end{equation}
Here, using the power-law at the critical point (at last!) we see that
\[
 \dfrac{|I_{n+1}^1|}{|I_n^1|} \leq \gamma_0 \left(\dfrac{|I_{n+1}|}{|I_n|}\right)^{s_0}\ ,
\]
where $\gamma_0>0$ is the constant given in Definition \ref{def:multicritic} and $s_0>1$ is the criticality of the 
critical point $c_0$. Carrying this information back to \eqref{crithelps4} gives us
\begin{equation}\label{crithelps5}
 b(I_n^1,f(T)) \leq \theta_2 \left(\dfrac{|I_{n+1}|}{|I_n|}\right)^{s_0}\ ,
\end{equation}
where $\theta_2=\gamma_0C_3/(1+ C_3)$. Combining \eqref{crithelps3} and \eqref{crithelps5} we get the inequality
\[
 \dfrac{|I_{n+1}|}{|I_n|} \geq \left(\dfrac{\theta_1}{B\theta_2}\right)^{\frac{1}{s_0-1}} = \theta_3\ .
\]
Summarizing, we have proved that
\begin{equation}\label{crithelps6}
 \theta_3|I_n| \leq |I_{n+1}|\leq C_3|I_n| \ .
\end{equation}
\item [(v)] {\it The intervals $I_{n}$ and $I_{n+1}^{q_{n}}$ are comparable\/}. Note that $I_{n+1}^{q_{n}}\subset I_n$, so 
$|I_{n+1}^{q_{n}}|\leq |I_n|$. We must prove an inequality in the opposite direction. 
For this purpose, let us consider the
interval $T^*=I_{n+1}^{q_n}\cup I_n^{q_n}\cup I_n^{2q_n}$. We shall look at the cross-ratio distortion of the pair 
$(I_n^{q_n}, T^*)$ under the map $f^{q_{n+1}-q_n}$. 
Clearly, the family $\{T^*,f(T^*),\ldots,f^{q_{n+1}-q_n}(T^*)\}$ has intersection 
multiplicity equal to at most $3$. By the Cross-Ratio Inequality, we have
\begin{equation}\label{crithelps7}
 D(f^{q_{n+1}-q_n};\,I_n^{q_n}, T^*)\;=\;\dfrac{b(I_n^{q_{n+1}}, f^{q_{n+1}-q_n}(T^*))}{b(I_n^{q_n}, T^*)}\;\leq\; B
\end{equation}
Now, the intervals $I_{n+1}^{q_{n+1}}$ and $I_{n+1}$ are dynamically symmetric with respect to their common endpoint $f^{q_{n+1}}(c_0)$. 
Also, the intervals $f^{q_{n+1}-q_n}(I_n^{2q_n})=I_n^{q_{n+1}+q_n}$ and $I_n^{q_{n+1}}$ are dynamically symmetric with respect to their 
common endpoint $f^{q_{n+1}+q_n}(c_0)$. Moreover, we have $I_n^{q_{n+1}}\subset I_n\cup I_{n+1}$. Combining these facts 
with \eqref{crithelps6} and Lemma \ref{lemma3}, we deduce after some computations that
\begin{align}\label{crithelps8}
 b(I_n^{q_{n+1}}, f^{q_{n+1}-q_n}(T^*))&= \dfrac{|I_{n+1}^{q_{n+1}}|}{|I_{n+1}^{q_{n+1}}|+|I_n^{q_{n+1}}|} 
\dfrac{|I_n^{q_{n+1}+q_n}|}{|I_n^{q_{n+1}}|+|I_n^{q_{n+1}+q_n}|} \nonumber  \\ 
&\geq \dfrac{C_3^{-2}\theta_3}{(1+C_3)(1+C_3+C_3^2)} \ .
\end{align}
We proceed to bound the denominator in \eqref{crithelps7} from above in similar fashion. 
Since the intervals $I_n^{q_n}$ and $I_n^{2q_n}$ are dynamically symmetric with respect to their common endpoint $f^{q_n}(c_0)$, applying 
Lemma \ref{lemma3} one final time yields
\begin{align}\label{crithelps9}
 b(I_n^{q_n}, T^*)\;&=\; \dfrac{|I_{n+1}^{q_{n}}|}{|I_{n+1}^{q_{n}}|+|I_n^{q_{n}}|} 
\dfrac{|I_n^{2q_n}|}{|I_n^{q_{n}}|+|I_n^{2q_n}|} \leq 
\dfrac{|I_{n+1}^{q_{n}}|}{|I_n^{q_{n}}|} \dfrac{C_3}{1+C_3^{-1}}\nonumber \\
&\leq\; \dfrac{C_3^2}{1+C_3^{-1}}\dfrac{|I_{n+1}^{q_{n}}|}{|I_n|} \ .
\end{align}
Putting \eqref{crithelps8} and \eqref{crithelps9} back into \eqref{crithelps7}, we deduce at last that
\begin{equation}\label{crithelps10}
 \theta_4 |I_n|\;\leq\; |I_{n+1}^{q_n}| \leq |I_n| \ ,
\end{equation}
where
\[
 \theta_4=\dfrac{(1+ C_3^{-1})C_3^{-4}\theta_3}{B(1+C_3)(1+C_3+C_3^2)}
\]
\end{enumerate}
The above estimates -- more precisely the inequalities \eqref{ineqa1}, \eqref{ineqa2}, \eqref{ineqb3}, \eqref{crithelps6} and \eqref{crithelps10} -- 
provide bounds for 5 of the 15 comparability ratios involved in \eqref{labellem4}. 
Each of the remaining 10 comparability ratios is obtained by suitable telescoping products of at most 4 of these 5 ratios. 
Thus, define $K$ to be the largest of all constants greater than $1$ appearing as bounds in 
the above estimates, namely $K=\max\{C_3(C_3+1), \theta_3^{-1}, \theta_4^{-1}  \}$. With this choice, all 15 inequalities involved in 
\eqref{labellem4} are established provided we take $C_4=K^4$.  
\end{proof}

\subsection{Proof of Theorem \ref{realbounds}}

Finally, to obtain Theorem \ref{realbounds}, we use the Cross-Ratio Inequality to propagate the information in Proposition \ref{lemma4} 
to any pair of adjacent intervals in the dynamical partition $\mathcal{P}_n$. Let $M\in \mathcal{P}_n$, and let $L, R\in \mathcal{P}_n(c_0)$ be 
its two immediate neighbors; write $T=L\cup M\cup R$. It suffices to show that the $b$-cross-ratio $b(M,T)$ is bounded from below by a constant depending only 
on the constant $C_4$ of Proposition \ref{lemma4}. There are two cases to consider, depending on whether $M$ is a {\it short\/} or a {\it long\/} atom of 
the dynamical partition $\mathcal{P}_n$. If $M$ is a short atom, say $M=I_{n+1}^j$ with $j<q_n$, then $L$ and $R$ are both long atoms. 
In fact, the combinatorics tells us that one of them, say $R$, is the interval $I_n^j$, whereas the other, $L$, is the interval $I_n^{j+q_{n+1}-q_n}$.
But then the homeomorphism $f^{q_n-j}$ maps $M$ onto $M^*=I_{n+1}^{q_n}$ and $T$ onto $T^*=I_n^{q_{n+1}}\cup I_{n+1}^{q_n}\cup I_{n}^{q_n}$. 
By Proposition \ref{lemma4}, the cross-ratio $b(M^*,T^*)$ is bounded from below (by a constant depending only on $C_4$).
Since the intervals $T,f(T),\ldots,f^{q_n-j}(T)=T^*$ have multiplicity of intersection at most $3$, it follows from the Cross-Ratio Inequality 
that 
\[
D(f^{q_n-j};M,T)\;=\; \frac{b(M^*,T^*)}{b(M,T)}\;\leq\; B \ .
\] 
Therefore $b(M,T)$ is also bounded from below (by a constant depending only on $C_4$). The same argument applies, {\it mutatis mutandis\/}, when 
$M$ is a long atom. This finishes the proof. 

\subsection{On the notion of comparability}\label{comparenotion}

We have presented the proof of the real bounds in such a way as to easily keep track of the constants involved in all the estimates -- so that 
one could actually write down the constant $C$ in Theorem \ref{realbounds} explicitly. 
From this point on, however, rather than continuing to keep track of such constants, 
we will use the same notion and notation of comparability introduced in \cite{dFdM}. Namely, 
given two positive real numbers $\alpha$ and $\beta$, we will say that $\alpha$ is {\it comparable\/} to $\beta$ {\it modulo\/} $f$ 
(or simply that they are {\it comparable\/}) if there exists a constant $K>1$ depending only on $C=C(f)$ such that 
$K^{-1}\beta\leq \alpha\leq K\beta$. This relation will be denoted $\alpha\asymp \beta$. As observed in 
\cite[p.~350]{dFdM}, comparability modulo $f$ is reflexive and symmetric but not transitive: if we are given a comparability chain 
$\alpha_1\asymp \alpha_2\asymp\cdots \asymp \alpha_k$, 
we can only say that $\alpha_1\asymp \alpha_k$ if the length $k$ of the chain is bounded by a constant that depends only on $f$. 
In everything we do in this paper, the lengths of all comparability chains are in fact universally bounded.  

\subsection{Beau bounds} \label{sec:beau}

In the study of renormalization of one-dimensional dynamical systems -- especially when pursuing Sullivan's strategy 
as outlined in \S \ref{sec:intro} -- one usually tries to get  
bounds which are {\it asymptotically universal\/}. In other words, in the context of renormalization it is desirable to know whether the constant 
$C=C(f)$ in Theorem \ref{realbounds} can, for all sufficiently large $n$, be replaced by a universal constant. Bounds of this type are called 
{\it beau\/} by Sullivan in \cite{Su}. We do not attempt to prove in the present paper that our bounds are {\it beau\/}, since this property is not 
relevant for our purposes.
  
\section{Geometry of dynamical partitions}\label{sec:geom}

In this section we present some geometric consequences of the real bounds that will be crucial in the proof of our main theorem.
The results below refer to the dynamical partitions $\mathcal{P}_n(c_k)$ ($0\leq k\leq N-1$, $n\in \mathbb{N}$) of a multicritical circle map $f$ for which the real 
bounds of Theorem \ref{realbounds} hold true. Recall that the atoms of each partition $\mathcal{P}_n(c_k)$ are of two types: the {\it long\/} atoms, {\it i.e.\/} 
those of the form $I_n^i(c_k)$, $0\leq i<q_{n+1}$, and the {\it short\/} atoms, {\it i.e.\/} 
those of the form $I_{n+1}^j(c_k)$, $0\leq j<q_{n}$. In what follows, we use the notion (and notation) of comparability 
introduced in \S \ref{comparenotion}. 

\subsection{Intersecting atoms are comparable}
The first result states that any two intersecting atoms belonging to dynamical partitions of two distinct critical points at the same level $n$ 
are comparable.

\begin{lemma}\label{intersectcomp}
 Let $c, c'$ be any two critical points of our map $f$. If $\Delta\in \mathcal{P}_n(c)$ and $\Delta'\in \mathcal{P}_n(c')$ are two atoms 
such that $\Delta\cap \Delta'\neq \O$, then $|\Delta|\asymp |\Delta'|$, i.e. they are comparable.
\end{lemma}

\begin{proof}
 Let $C=C(f)>1$ be the constant given by the real bounds (Theorem \ref{realbounds}). There are three cases to consider, according to the 
types of atoms we have: long/long, long/short, and short/short. More precisely, we have the following three cases.
\begin{enumerate}
 \item[(i)] We have $\Delta=I_n^i(c)$ and $\Delta'=I_n^j(c')$, where $0\leq i,j <q_{n+1}$. 
Here we may assume that $f^j(c')\in \Delta= [f^i(c),f^{i+q_n}(c)]$. Then $f^{i+q_n}(c)\in \Delta'=[f^j(c'),f^{j+q_n}(c')]$, 
and we have the situation depicted in Figure \ref{intercomp}($a$). Using the monotonicity of $f^{q_n}$, we see that 
$\Delta'\subset \Delta\cup f^{q_n}(\Delta)$. Applying Lemma \ref{lemma3} to $x=f^{i+q_n}(c)$, we see that 
$\Delta=[f^{-q_n}(x),x]$ and $f^{q_n}(\Delta)=[x, f^{q_n}(x)]$ satisfy $|f^{q_n}(\Delta)|\leq C|\Delta|$, and from this 
it follows that $|\Delta'|\leq (1+C)|\Delta|$. Conversely, we also have $\Delta\subset f^{-q_n}(\Delta')\cup \Delta'$. 
Again applying Lemma \ref{lemma3}, this time to $x=f^j(c')$, we deduce just as before that $|f^{-q_n}(\Delta')|\leq C|\Delta'|$, 
and therefore $|\Delta|\leq (1+C)|\Delta'|$. Hence $\Delta$ and $\Delta'$ are comparable in this case.

\item[(ii)] We have $\Delta=I_n^i(c)$ and $\Delta'=I_{n+1}^j(c')$, where $0\leq i <q_{n+1}$ and $0\leq j<q_n$. 
Here, we look at the interval $I_{n+1}^{i+q_n}(c)\subset \Delta$. This interval shares an endpoint with $\Delta$ 
(namely $f^{i+q_n}(c)$) and it is also an atom of $\mathcal{P}_{n+1}(c)$. In particular, $|I_{n+1}^{i+q_n}(c)|\asymp |\Delta|$, 
by the real bounds. There are now two sub-cases. If $\Delta'\cap I_{n+1}^{i+q_n}(c) \neq \O$, then, since 
$\Delta'$ also belongs to $\mathcal{P}_{n+1}(c')$, case (i) above tells us that $|\Delta'|\asymp |I_{n+1}^{i+q_n}(c)|$, and therefore 
$\Delta'$ is comparable to $\Delta$ in this sub-case. On the other hand, if $\Delta'\cap I_{n+1}^{i+q_n}(c) = \O$, then 
we must have $f^j(c')\in \Delta$ (see Figure \ref{intercomp}($b$)). In this sub-case, we consider the interval $I_n^j(c')\in \mathcal{P}_n(c')$, 
which also has $f^j(c')$ as an endpoint. Then we have $\Delta \cap I_n^j(c')\neq \O$, and again by case (i) we have 
$|\Delta|\asymp |I_n^j(c')|$. But by the real bounds we have $|I_n^j(c')|\asymp |I_{n+1}^j(c')|=|\Delta'|$, so 
$\Delta'$ is comparable to $\Delta$ also in this sub-case. 

\item[(iii)] We have $\Delta=I_{n+1}^i(c)$ and $\Delta'=I_{n+1}^j(c')$, where $0\leq i,j <q_{n}$. 
This case is entirely analogous to case (i). 

\end{enumerate}

\begin{figure}[t]
\begin{center}
\psfrag{i1}[][]{ $\Delta=I_n^i(c)$} 
\psfrag{i2}[][][1]{$f^{q_n}(\Delta)$}
\psfrag{i3}[][][1]{$f^{-q_n}(\Delta')$} 
\psfrag{i4}[][][1]{$\Delta'=I_n^j(c')$}
\psfrag{i5}[][][1]{$I_{n+1}^{i+q_n}(c)$}
\psfrag{i6}[][][1]{$\Delta'\!=\!I_{n+1}^j\!(c')$}
\psfrag{i7}[][][1]{$I_n^j(c')$}
\psfrag{a}[][][1]{$(a)$}
\psfrag{b}[][][1]{$(b)$}
\includegraphics[width=4.0in]{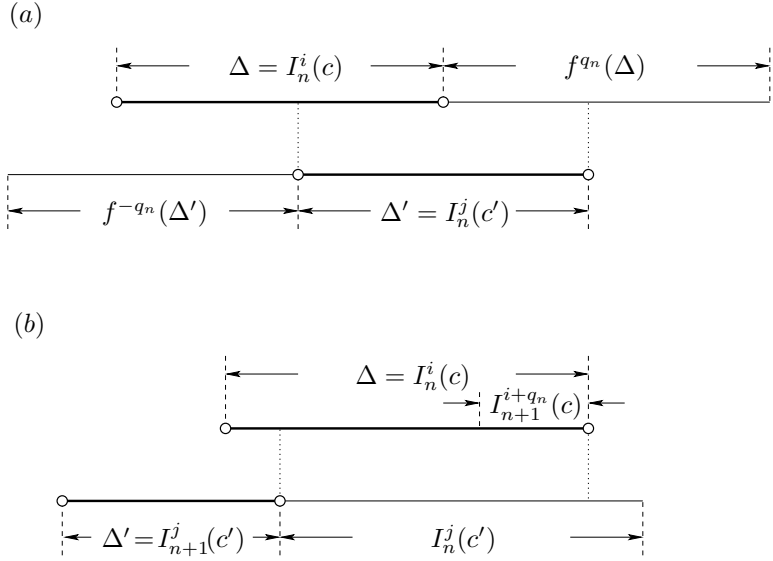}
\end{center}
\caption[intercomp]{\label{intercomp} The cases long/long and long/short of Lemma \ref{intersectcomp}.}
\end{figure}

\end{proof}

\subsection{Critical atoms are large}
Let us now consider the first return map to the interval $I_n(c_0)\cup I_{n+1}(c_0)$, or equivalently the pair of maps 
$f^{q_n}|_{I_{n+1}(c_0)}\,,\,f^{q_{n+1}}|_{I_{n}(c_0)}$. Besides $c_0$ (which is critical for both maps in the pair), 
this return map has at most $N-1$ other critical points: some in $I_n(c_0)$, and some in $I_{n+1}(c_0)$. 
Our next auxiliary result states that the intervals of the dyamical partition at the next level ($\mathcal{P}_{n+1}(c_0)$) 
which contain these critical points of the return map at level $n$ must be comparable with their parent atom ($I_n(c_0)$ or 
$I_{n+1}(c_0)$). 

\begin{lemma}\label{criticlarge}
 Let $0\leq k< a_{n+1}$ be such that the interval $f^{q_n+kq_{n+1}}(I_{n+1}(c_0))\subset I_n(c_0)$ contains a critical point 
of $f^{q_{n+1}}$. Then 
\begin{equation}\label{largecritic}
 \left|f^{q_n+kq_{n+1}}(I_{n+1}(c_0))\right|\asymp \left|I_n(c_0)\right| \ .
\end{equation}
\end{lemma}

\begin{proof}
 If $k=0$ there is nothing to prove, since we already know from the real bounds that 
$|f^{q_n}(I_{n+1}(c_0))|\asymp |I_n(c_0)|$. Hence we assume that $1\leq k \leq a_{n+1}-1$.  
Let us write $\Delta=f^{q_n+kq_{n+1}}(I_{n+1}(c_0))$ in this proof. Let $0<j\leq q_{n+1}$ be such that 
$f^j(\Delta)\ni c_1$, where $c_1\neq c_0$ is another critical point of $f$. Note that 
$I_n^j(c_0)=f^j(I_n(c_0))\supset f^j(\Delta)$. 
We claim that $|f^j(\Delta)|\asymp |f^j(I_n(c_0))|$. This is a consequence of the following two facts.
\begin{enumerate}
 \item[(i)] We have $|I_n^j(c_0)|\asymp |I_{n+1}(c_1)|$. Indeed, these two intervals have non-empty intersection 
(they both contain $c_1$), and since $I_n^j(c_0)\in \mathcal{P}_n(c_0)$ and $I_{n+1}(c_1)\in \mathcal{P}_n(c_1)$, 
their comparability follows from Lemma \ref{intersectcomp}. 
 \item[(ii)] We have $|I_{n+1}(c_1)|\asymp |f^j(\Delta)|$. To see why, first note that 
\[
 j+q_n+kq_{n+1}\leq q_n+(k+1)q_{n+1} \leq q_n+a_{n+1}q_{n+1} = q_{n+2}\ ,
\]
from which it follows that 
\[
 f^j(\Delta)= I_{n+1}^{j+q_n+kq_{n+1}}(c_0) \in \mathcal{P}_{n+1}(c_0)\ .
\]
Since $I_{n+1}(c_1)\in \mathcal{P}_{n+1}(c_1)$, and $f^j(\Delta)\cap I_{n+1}(c_1)\supset \{c_1\}\neq \O$, 
we may again apply Lemma \ref{intersectcomp} to deduce that $I_{n+1}(c_1)$ and $f^j(\Delta)$ are comparable. 
\end{enumerate}

With the claim at hand, we proceed as follows. Consider the (closure of the) gap between $\Delta$ and $I_{n+1}^{q_n}$ 
inside $I_n(c_0)$, namely the interval $J=\bigcup_{i=1}^{k-1} I_{n+1}^{q_n+iq_{n+1}}(c_0)$. Note that if $k=1$ then $J=\O$; in this case 
$\Delta$ and $I_{n+1}^{q_n}$  are two adjacent atoms of $\mathcal{P}_{n+1}(c_0)$, hence they are comparable by the real bounds 
(Theorem \ref{realbounds}) and there is nothing to prove. Therefore we assume that $k\geq 2$, so that $J\neq \O$. 
We already know from the above claim that $|f^j(\Delta)|\asymp |I_n^j(c_0)|$, and the real bounds also tell us that 
$|I_n^j(c_0)|\asymp|I_{n+1}^{j+q_n}(c_0)|$. Moreover, we have $I_{n+1}^{j+q_n+q_{n+1}}(c_0)\subseteq f^j(J)\subset I_n^j(c_0)$. Since 
$|I_{n+1}^{j+q_n+q_{n+1}}(c_0)|\asymp |I_{n+1}^{j+q_n}(c_0)|$, because these two intervals are consecutive atoms of $\mathcal{P}_{n+1}(c_0)$, 
it follows that $|f^j(J)|\asymp |I_{n+1}^{j+q_n}(c_0)|$. In other words, the consecutive intervals $f^j(\Delta)$, $f^j(J)$ and $I_{n+1}^{j+q_n}(c_0)$ 
are pairwise comparable. In particular, the $b$-cross-ratio determined by these three intervals is bounded from above and from below, 
{\it i.e.\/} there exists a constant $K>1$ depending only on the constant $C$ of the real bounds such that
\begin{equation}\label{JTcross}
 K^{-1}\leq b(f^j(J),f^j(T)) \leq K\ .
\end{equation}
Here we have written $T=\Delta \,\cup\, J \,\cup\, I_{n+1}^{q_n}(c_0)$. Note that 
$T, f(T),\ldots, f^j(T)$ are pairwise disjoint. Therefore, by the cross-ratio inequality applied to the 
homeomorphism $f^j$ (and $m=1$), we have $D(f^j; J, T)\leq C$, or equivalently $b(f^j(J), f^j(T))\leq C b(J,T)$. 
Using the lower estimate in \eqref{JTcross}, we see that $b(J,T)\geq C^{-1}K^{-1}$, that is,
\begin{equation}\label{JTcross2}
 \frac{|\Delta|\,|I_{n+1}^{q_n}(c_0)|}{|\Delta\,\cup\, J|\,|J\,\cup\, I_{n+1}^{q_n}(c_0)|} \geq (CK)^{-1}\ .
\end{equation}
But, since $J\supseteq I_{n+1}^{q_n+q_{n+1}}(c_0)$, and since $I_{n+1}^{q_n+q_{n+1}}(c_0)$ and $I_{n+1}^{q_n}(c_0)$ are adjacent atoms 
of $\mathcal{P}_{n+1}(c_0)$, we have by the real bounds
\[
 |\Delta \cup J|>|J|\geq |I_{n+1}^{q_n+q_{n+1}}(c_0)| \geq  C^{-1} |I_{n+1}^{q_n}(c_0)|\ .
\]
Moreover, $|I_{n+1}^{q_n}(c_0)|\geq C^{-1}|I_n(c_0)|$, again by the real bounds. 
Putting these facts back into \eqref{JTcross2}, we deduce that
\[
 |\Delta| \geq C^{-2}K^{-1} |J\,\cup\, I_{n+1}^{q_n}(c_0)| > C^{-3}K^{-1} |I_n(c_0)|\ .
\]
This shows that $\Delta$ and $I_n(c_0)$ are comparable. Hence \eqref{largecritic} is established, and the proof 
of Lemma \ref{criticlarge} is complete. 
\end{proof}

\subsection{Building an auxiliary partition}\label{auxpart}
In this subsection we will construct a suitable refinement of the dynamical partition $\mathcal{P}_n(c_0)$ (for each $n\geq1$). This auxiliary partition, 
which we denote by $\mathcal{P}_n^*(c_0)$, is finer than $\mathcal{P}_n(c_0)$ but coarser than $\mathcal{P}_{n+1}(c_0)$. 
Such auxiliary partition will be needed in the construction of the fine grid of \S \ref{secmainthm}. 

From now on we write, for $0\leq k < a_{n+1}$, $\Delta_k=f^{q_n+kq_{n+1}}(I_{n+1}(c_0))$. Note that each 
$\Delta_k$ is an atom of the dynamical partition $\mathcal{P}_{n+1}(c_0)$, and that 
\[\bigcup_{k=0}^{a_{n+1}-1} \Delta_k= I_n(c_0)\setminus I_{n+2}(c_0) \ .\]
We consider the times $0\leq k_1<k_2<\cdots <k_r<a_{n+1}$ having the property that $\Delta_{k_i}$ contains a critical point 
of $f^{q_{n+1}}$. These are called the {\it critical times\/} at level $n$. For convenience of notation, we also 
define $k_0=0$. Note that $f^{q_{n+1}}$ has at most $N$ critical points in $I_n(c_0)$, where $N$ is the total number of critical points of $f$. 
Since each such critical point belongs to at most two of the $\Delta_k$'s, we see that  
$r\leq 2N$. Thus, although the non-negative integer $r$ may depend on $n$ (the level of renormalization), it nevertheless ranges over only finitely 
many values. The critical times $k_i$ also depend on $n$. The intervals $\Delta_{k_i}$ for $0\leq i\leq r$ will be called 
{\it critical spots\/}. 

For each $i=0,1,\ldots,r-1$, let $G_i\subseteq I_n(c_0)\setminus I_{n+2}(c_0)$ be the gap between the 
two consecutive critical spots $\Delta_{k_i}$ and $\Delta_{k_{i+1}}$ inside 
$I_n(c_0)$, namely the interval
\[
 G_i=\bigcup_{k=k_i+1}^{k_{i+1}-1} \Delta_k \ .
\]
We also define, for $i=r$, 
\[
 G_r=\bigcup_{k=k_r+1}^{a_{n+1}-1} \Delta_k \ .
\]
We call $G_i$ the $i$-th {\it bridge\/} of $I_n(c_0)$. See figure \ref{bridgespots}. 
We remark that it may well be the case that $G_i=\O$ for some (or all!) values of $i$. 

\begin{figure}[h]
\begin{center}
\psfrag{c}[][]{ $c_0$} 
\psfrag{I1}[][][1]{$I_{n+2}(c_0)$}
\psfrag{I}[][][1]{$I_{n}(c_0)$} 
\psfrag{f}[][][1]{$\;\;f^{q_{n}}(c_0)$}
\psfrag{d0}[][][1]{$\Delta_0$}
\psfrag{d1}[][][1]{$\Delta_{k_i}$}
\psfrag{d2}[][][1]{$\Delta_{k_{i+1}}$}
\psfrag{dr}[][][1]{$\Delta_{k_r}$}
\psfrag{dt}[][][1]{$\cdots$}
\psfrag{Gi}[][]{$G_i$} 
\psfrag{Gr}[][]{$G_r$}
\includegraphics[width=4.5in]{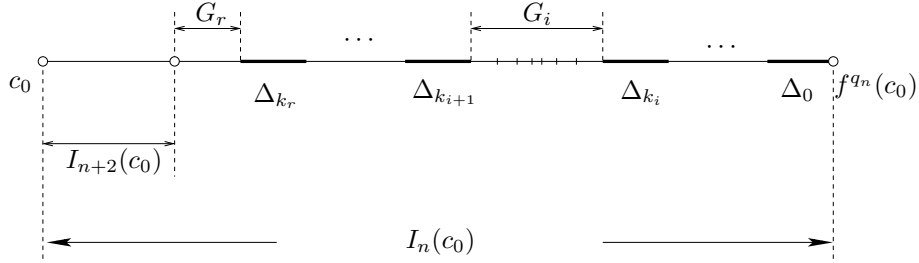}
\end{center}
\caption[slope]{\label{bridgespots} Bridges and critical spots.}
\end{figure}

\begin{lemma}\label{largebridge}
 Each non-empty bridge $G_i$ is comparable to $I_n(c_0)$.
\end{lemma}

\begin{proof}
If $G_i\neq \O$, then $G_i$ contains at the very least the atom $\Delta_{k_i+1}$, 
adjacent to $\Delta_{k_i}$, and so we have $|G_i|\geq |\Delta_{k_i+1}|\asymp |\Delta_{k_i}|$, by the real bounds. 
By Lemma \ref{criticlarge}, we have $|\Delta_{k_i}|\asymp |I_n(c_0)|$.
Since we also have $G_i\subset I_n(c_0)$, it follows that $|G_i|\asymp |I_n(c_0)|$.  
\end{proof}

Thus, we have the following decomposition of $I_n(c_0)\setminus I_{n+2}(c_0)$ as union of at most $2r+2\leq 4N+2$ intervals:
\begin{equation}\label{decomp}
 I_n(c_0)\setminus I_{n+2}(c_0)\;=\; \bigcup_{i=0}^{r}\Delta_{k_i} \cup \bigcup_{i=0}^{r}G_i \ .
\end{equation}
In view of Lemmas \ref{criticlarge} and \ref{largebridge}, as well as the real bounds, each interval 
in the above decomposition is comparable to $I_n(c_0)$. In particular, they are all pairwise comparable. 

\begin{remark}\label{critspotremark}
Note that the image of each critical spot $\Delta_{k_i}$ under $f^{q_{n+1}}$ is also comparable to $I_n(c_0)$: this 
is simply because $f^{q_{n+1}}(\Delta_{k_i})=\Delta_{k_i+1}$ is adjacent to $\Delta_{k_i}$ in $\mathcal{P}_{n+1}(c_0)$. 
Likewise, the image of each bridge $G_i$ under $f^{q_{n+1}}$ is also comparable to $I_n(c_0)$, because either $i<r$ and 
$f^{q_{n+1}}(G_i)$ contains the critical spot $\Delta_{k_{i+1}}$, or $i=r$, in which case $f^{q_{n+1}}(G_r)$ contains $I_{n+2}(c_0)$. 
\end{remark}

Let us now map the decomposition \eqref{decomp} forward by $f$ to get corresponding decompositions 
of all long atoms $I_n^j(c_0)\in \mathcal{P}_n(c_0)$, for $j=1,2,\ldots,q_{n+1}-1$. 
We get in this fashion a new partition $\mathcal{P}_{n}^*(c_0)$ of the circle (modulo endpoints). 
More precisely, let
\begin{align}\label{pstar}
 \mathcal{P}_{n}^*(c_0) \;=\;& \,\left\{f^j(\Delta_{k_i}): 0\leq i\leq r\;;\; 0\leq j\leq q_{n+1}-1\right\} \\
                        &\cup \left\{f^j(G_i): 0\leq i\leq r\;;\; 0\leq j\leq q_{n+1}-1 \right\} \nonumber \\
                        &\cup \left\{f^j(I_{n+2}): 0\leq j\leq q_{n+1}-1 \right\} \nonumber \\
                        &\cup \left\{f^{\ell}(I_{n+1})\,:\; 0\leq \ell\leq q_{n}-1        \right\} \ .\nonumber
\end{align}
This partition 
refines $\mathcal{P}_{n}(c_0)$, although not strictly because each short atom of $\mathcal{P}_{n}(c_0)$ is left untouched 
by the above procedure. Generalizing the nomenclature introduced earlier, all atoms of $\mathcal{P}_{n}^*(c_0)$ of 
the form $f^j(\Delta_{k_i})$ are called {\it critical spots\/}, and all those of the form $f^j(G_i)$ are called 
{\it bridges\/}. 

\begin{proposition}\label{consecatoms}
 Any two consecutive atoms of $\mathcal{P}_{n}^*(c_0)$ are comparable.
\end{proposition}

\begin{proof}
By the real bounds (Theorem \ref{realbounds}), the partition $\mathcal{P}_{n}(c_0)$ has the stated comparability property. 
Hence it suffices to check that all bridges and critical spots of $\mathcal{P}_{n}^*(c_0)$ 
inside each long atom $I_n^j(c_0)\in \mathcal{P}_{n}(c_0)$ 
are comparable to $I_n^j(c_0)$. We already know this for $j=0$ (see Lemma \ref{largebridge} and the paragraph following its proof). 
For the other values of $j$, map $I_n^j(c_0)$ forward by $f^{q_{n+1}-j}$ onto $I_n^{q_{n+1}}(c_0)\subset I_n(c_0)\cup I_{n+1}(c_0)$ 
and apply the Cross-Ratio Inequality, combined with Remark \ref{critspotremark}.

\end{proof}

\subsection{Almost parabolic maps} In order to construct the fine grid in section \ref{secmainthm} below, 
we not only will need to use whole atoms from the various partitions $\mathcal{P}_{n}^*(c_0)$, but also 
will have to break some of these atoms even further, in a suitable way. In this subsection and the next, we show how to 
break such atoms as required. 

First we need to recall the definition of an almost parabolic map, as given in \cite[\S 4.1]{dFdM}.  

\begin{definition}
 An \emph{almost parabolic map} is a $C^3$ diffeomorphism 
\[
 \phi:\, J_1\cup J_2\cup \cdots \cup J_\ell \;\to\; J_2\cup J_3\cup \cdots \cup J_{\ell+1} \ ,
\]
where $J_1,J_2, \ldots, J_{\ell+1}$ are consecutive intervals on the circle (or on the line), with the 
following properties. 
\begin{enumerate}
 \item[(i)] One has $\phi(J_\nu)= J_{\nu+1}$ for all $1\leq \nu\leq \ell$;
 \item[(ii)] The Schwarzian derivative of $\phi$ is everywhere negative.
\end{enumerate}
The positive integer $\ell$ is called the \emph{length} of $\phi$, and the positive real number 
\[
 \sigma =\min\left\{\frac{|J_1|}{|\cup_{\nu=1}^\ell J_\nu|}\,,\, \frac{|J_\ell|}{|\cup_{\nu=1}^\ell J_\nu|}     \right\}
\]
is called the \emph{width\/} of $\phi$.
\end{definition}

Note that the dynamics of an almost parabolic $\phi$ is, rather trivially, conjugate to a translation, and  
each interval $J_\nu$ is a {\it fundamental domain\/} for $\phi$ (modulo endpoints). Now, the crucial fact about the geometry 
of the fundamental domains of an almost parabolic map is the following lemma, due to Yoccoz.

\begin{lemma}[Yoccoz]\label{lemyoccoz}
 Let $\phi: \bigcup_{\nu=1}^\ell J_\nu \to \bigcup_{\nu=2}^{\ell+1} J_\nu$ be an almost parabolic map with length $\ell$ and 
width $\sigma$. There exists a constant $C_\sigma>1$ (depending on $\sigma$ but not on $\ell$) such that, 
for all $\nu=1,2,\ldots,\ell$, we have
\begin{equation}\label{yocineq}
 \frac{C_\sigma^{-1}|I|}{[\min\{\nu,\ell+1-\nu\}]^2} \;\leq\; |J_\nu| \;\leq\;  \frac{C_\sigma|I|}{[\min\{\nu,\ell+1-\nu\}]^2}\ ,
\end{equation}
where $I=\bigcup_{\nu=1}^\ell J_\nu$ is the domain of $\phi$. 
\end{lemma}

A proof of this lemma can be found in \cite[Appendix B]{dFdM}. Defining the {\it order\/} of a fundamental domain $J_\nu$ 
as above to be $\mathrm{ord}(J_\nu)=\min\{\nu,\ell+1-\nu\}$, we can rephrase the conclusion of Lemma \ref{lemyoccoz} 
as follows: for all $\nu=1,2,\ldots,\ell$, we have $|J_\nu|\asymp (\mathrm{ord}(J_\nu))^{-2}|I|$ with comparability 
constant depending only on $\sigma$. In other words, the relative size of a fundamental domain in an almost parabolic map 
is inversely proportional to the square of its order. 

The following lemma exhibits a special way of grouping together the fundamental domains of an almost parabolic map. 

\begin{lemma}\label{balancedecomp}
 Let $\phi$ be an almost parabolic map with domain $I=\bigcup_{\nu=1}^\ell J_\nu$, and let $d\in \mathbb{N}$ be largest 
such that $2^{d+1}\leq \ell/2$. There exists a descending chain of (closed) intervals (see Figure \ref{balancedom}) 
\[
 I=M_0\supset M_1\supset \cdots \supset M_{d+1}
\]
for which, letting $L_i,R_i$ denote the (left and right) connected components of $M_i\setminus M_{i+1}$ for all 
$0\leq i\leq d$, the following properties hold.
\begin{enumerate}
 \item[(i)] Each of the intervals $L_i,R_i$ is the union of exactly $2^i$ consecutive atoms (fundamental domains) 
of $I$.
 \item[(ii)] We have
\begin{equation}\label{balance1}
 I\;=\; \bigcup_{i=0}^{d}L_i\;\cup\;M_{d+1}\;\cup\; \bigcup_{i=0}^{d}R_i \ . 
\end{equation}
 \item[(iii)] For each $0\leq i\leq d$ we have $|L_i|\asymp |M_{i+1}|\asymp |R_i|$, with comparability constants 
depending only on the width $\sigma$ of $\phi$. 
\end{enumerate}
\end{lemma}
\begin{figure}[ht]
\begin{center}
\psfrag{l0}[][]{ $L_0$} 
\psfrag{l1}[][][1]{$L_1$}
\psfrag{r0}[][][1]{$R_0$} 
\psfrag{r1}[][][1]{$R_1$}
\psfrag{m1}[][][1]{$M_1$}
\psfrag{m2}[][][1]{$M_2$}
\psfrag{md}[][][1]{$M_d$}
\psfrag{md1}[][][1]{$M_{d+1}$}
\includegraphics[width=4.0in]{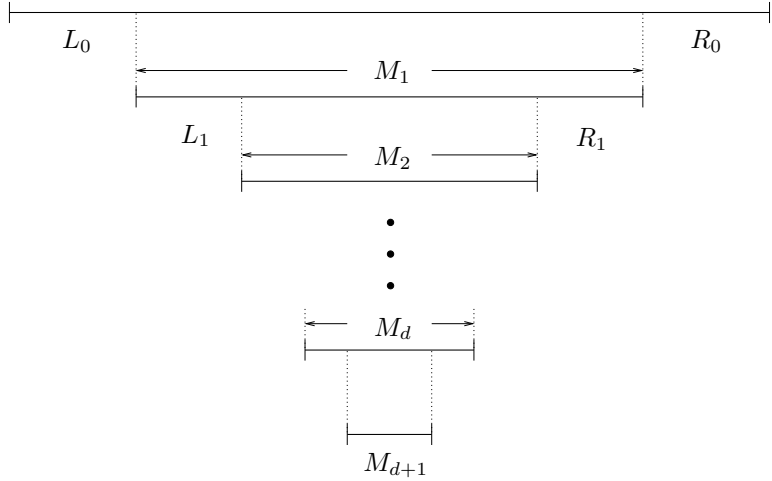}
\end{center}
\caption[balancedom]{\label{balancedom} Balanced decomposition of the domain of an almost parabolic map.}
\end{figure}
\begin{proof}
 We define, for each $0\leq i\leq d$,
\[
 L_i\;=\;\bigcup_{\nu=2^i}^{2^{i+1}-1}J_\nu \ \ ;\ \ R_i\;=\;\bigcup_{\nu=\ell+2-2^{i+1}}^{\ell+1-2^{i}}J_\nu \ .
\]
Also, for each $0\leq i\leq d+1$, we let
\[
 M_i\;=\;\bigcup_{\nu=2^i}^{\ell+1-2^{i}}J_\nu
\]
Then we immediately have (i) and (ii). Hence all we have to do is prove (iii). Let us fix $0\leq i\leq d$. 
In all that follows, the implicit comparability constants are either universal or depend on the constant $C_\sigma$ of 
Yoccoz's Lemma \ref{lemyoccoz}. Applying that lemma, we see that
\begin{equation}\label{balance2}
 |L_i|\;=\;\sum_{\nu=2^i}^{2^{i+1}-1} |J_\nu| \asymp \left(\sum_{\nu=2^i}^{2^{i+1}-1} \frac{1}{\nu^2}\right)|I| \asymp 2^{-i}|I|\ .
\end{equation}
Similarly, we have
\begin{equation}\label{balance3}
 |R_i|\asymp 2^{-i}|I|\ .
\end{equation}
Moreover, we can write
\begin{equation}\label{balance4}
 |M_{i+1}|\;=\;\sum_{\nu=2^{i+1}}^{2^{i+2}-1} |J_\nu| \asymp 2\left( \sum_{2^{i+1}\leq \nu\leq \frac{\ell}{2}} \frac{1}{\nu^2}\right)|I| = 2A|I|\ ,
\end{equation}
where the number $A$ satisfies
\begin{equation}\label{balance5}
 \sum_{\nu=2^{i+1}}^{2^{i+2}-1} \frac{1}{\nu^2} \leq A \leq \sum_{\nu=2^{i+1}}^{\infty} \frac{1}{\nu^2}\ .
\end{equation}
Both sums appearing in \eqref{balance5} are comparable to $2^{-i-1}$ (use the integral test).
Hence \eqref{balance4} and \eqref{balance5} put together yield
\begin{equation}\label{balance6}
 |M_{i+1}|\asymp 2^{-i}|I|\ .
\end{equation}
Combining \eqref{balance2}, \eqref{balance3} and \eqref{balance6}, we see that (iii) holds true as well, and the proof is complete. 
\end{proof}

\begin{remark}
Given an interval $I$ partitioned into atoms $J_\nu$, $1\leq \nu\leq \ell$, as above, a decomposition 
of the form \eqref{balance1} satisfying properties (i), (ii), (iii) of Lemma \ref{balancedecomp} is called 
a \emph{balanced decomposition\/} of $I$ (relative to its given partition into atoms). Thus, Lemma \ref{balancedecomp} can be re-stated as saying that the domain 
of an almost parabolic map always admits a balanced decomposition. In such balanced decomposition, the intervals 
$M_i$, $0\leq i\leq d+1$, are said to be \emph{central\/}, whereas the intervals $L_i,R_i$, $0\leq i\leq d$, are 
said to be \emph{lateral\/}. The positive integer $d$ is the \emph{depth\/} of the decomposition. 
\end{remark}

\begin{remark}\label{ratioofJs}
 The following fact, more general than what was used in the proof of Lemma \ref{balancedecomp}, holds true 
for the fundamental domains $J_\nu$ ($1\leq \nu\leq \ell$) of any almost parabolic map $\phi$: For all 
$1\leq k<l<m\leq \ell$, one has
\[
 \frac{|J_{l+1}|+|J_{l+2}|+\cdots+|J_m|}{|J_{k+1}|+|J_{k+2}|+\cdots+|J_l|}\;\asymp\; \frac{k(m-l)}{m(l-k)}\ ,
\]
with comparability constant depending only on the width $\sigma$ of $\phi$ {\footnote{In fact, the comparability constant 
can be taken to be equal to (a universal constant times) $C_\sigma^2$, where $C_\sigma$ is the constant in Lemma \ref{lemyoccoz}.}} 
Again, this follows from Yoccoz's Lemma \ref{lemyoccoz}. This fact will be useful in \S \ref{sec:finegrid}. 
\end{remark}

\begin{remark}\label{balremark}
 Let $I, I^*$ be two closed intervals with $I^*$ contained in the interior of $I$. Let $I^*$ be partitioned into a finite number $\ell$ 
of atoms, consecutively labelled $J_\nu$, $1\leq \nu\leq \ell$ as before, and suppose such atoms satisfy the inequalities \eqref{yocineq} 
(for some choice of the constant $C_\sigma$) 
-- so that we have a balanced decomposition of $I^*$ (as in Lemma \ref{balancedecomp}). Then, adding both lateral components of 
$I\setminus I^*$ to the collection of $J_\nu$'s and re-labelling these $\ell+2$ 
atoms from first to last, one sees that the inequalities \eqref{yocineq} hold true for the new collection also 
(with a different comparability constant, in general)   
Thus, we get a balanced decomposition of $I$ as well. This remark will be used in the proof of Corollary \ref{balancedbridges}. 
\end{remark}

\subsection{Balanced decompositions of bridges}
We distinguish two types of atoms belonging to the partition $\mathcal{P}_{n}^*(c_0)$:
\begin{enumerate}
 \item[(a)] {\it Regular atoms\/}: These consist of all short atoms of $\mathcal{P}_{n}(c_0)$, all of 
which belong to $\mathcal{P}_{n}^*(c_0)$, all intervals of the form $f^j(I_{n+2})$ (with $0\leq j\leq q_{n+1}-1$), all critical spots 
$f^j(\Delta_{k_i})$ (with $0\leq i\leq r$, $0\leq j\leq q_{n+1}-1$), together with all those 
bridges $G_{i,j}=f^j(G_i)$ that have less than $1,000$ atoms of $\mathcal{P}_{n+1}(c_0)$ in it ({\it i.e.}, those with
$k_{i+1}-k_i\leq 1,000$). 
 \item[(b)] {\it Saddle-node atoms\/}: These are the remaining bridges; to wit, those $G_{i,j}$ whose decomposition 
as a union of atoms of $\mathcal{P}_{n+1}(c_0)$ has at least $1,000$ such atoms in it ({\it i.e.\/}, those 
with $k_{i+1}-k_i > 1,000$). 
\end{enumerate} 

Proceeding by analogy with \cite[\S 4.3]{dFdM}, we will show, with the help of Yoccoz's Lemma \ref{lemyoccoz}, 
how to get  a {\it balanced decomposition\/} of a saddle-node bridge. 
Before we proceed, however, we make the following simplyfying assumption. 
Conjugating our multicritical circle map $f$ by a suitable $C^3$ diffeomorphism, we may 
assume without loss of generality that the map $f$  is \textit{canonical}, in the sense that each critical point $c_{k}$ has a 
neighborhood $\mathcal{U}_{k} \subseteq S^{1}$  
such that for all $x \in \mathcal{U}_{k}$
 \begin{equation*}
  f(x)=f(c_k)+ (x-c_{k})|x-c_{k}|^{s_{k}-1}
 \end{equation*}
where $s_{k}>1$ is the power-law of $c_{k}$ (as in \S \ref{sec:multi}). Note that this implies that the Schwarzian derivative $Sf$ 
is negative in each $\mathcal{U}_{k}$, {\it i.e.,} for all 
$x \in \mathcal{U}_{k} \setminus \{c_{k} \}$, we have
\begin{equation}\label{negS0}
 Sf(x)= -\dfrac{s_{k}^{2}-1}{2(x-c_{k})^{2}} <0 \ .
\end{equation}
We write $\mathcal{U}=\bigcup_{k=0}^{N-1} \mathcal{U}_k$, and we let $\mathcal{V}\subset S^1$ be an open set that contains  
none of the critical points of $f$ but is such that $\mathcal{U}\cup \mathcal{V}=S^1$. 

Now, consider a non-empty bridge $G_i\subset I_n(c_0)$, namely
\[
 G_i\;=\;\bigcup_{k=k_i+1}^{k_{i+1}-1} \Delta_k \ .
\]
We define the {\it reduced bridge\/} $G_i^*$ associated with $G_i$ to be
\[
 G_i^*\;=\;\bigcup_{k=k_i+2}^{k_{i+1}-2} \Delta_k \ .
\]
In other words, $G_i^*$ is simply $G_i$ minus its two lateral atoms. In particular, if $G_i$ is made up of $\leq  2$ atoms of $\mathcal{P}_{n+1}(c_0)$, then 
$G_i^*=\O$.  With this terminology, we can now state the 
following fundamental lemma.

\begin{lemma}\label{negSret}
 There exists a positive integer $n_0=n_0(f)$ such that the following holds for all $n\geq n_0$. 
For each non-empty reduced bridge $G_i^*\subset I_n(c_0)$, the restriction 
$f^{q_{n+1}}|_{G_i^*}$ has negative Schwarzian derivative everywhere, {\it i.e.,} for all 
$x\in G_i^*$ we have $Sf^{q_{n+1}}(x)<0$.   
\end{lemma}

\begin{proof}
 Given $x\in G_i^*$ and $n\in \mathbb{N}$, the chain rule for the Schwarzian derivative tells us that
\begin{equation}\label{negS1}
 Sf^{q_{n+1}}(x)\;=\; \sum_{j=0}^{q_{n+1}-1} Sf(f^j(x))\left[Df^j(x)\right]^2 \ .
\end{equation}
The sum on the right-hand side of \eqref{negS1} can be split as $\Sigma_1^{(n)}(x) + \Sigma_2^{(n)}(x)$ where
\begin{equation}\label{negS2}
 \Sigma_1^{(n)}(x)\;=\; \sum_{I_n^j(c_0)\subset \mathcal{U}} Sf(f^j(x))\left[Df^j(x)\right]^2 \ ,
\end{equation}
and $\Sigma_2^{(n)}(x)$ is the sum over the remaining terms. 

Let $\delta_n=\max_{0\leq j< q_{n+1}} |I_n^j(c_0)|$. We know that $\delta_n\to 0$ as $n\to \infty$, because $f$ is topologically conjugate 
to a rotation. Thus, choose $n_1=n_1(f)$ so large that $\delta_n$ is smaller than the Lebesgue number of the covering 
$\{\mathcal{U}, \mathcal{V}\}$ of the circle, for all $n\geq n_1$. Then we certainly have, for all $n\geq n_1$ and all $x\in G_i^*$,
\begin{equation}\label{negS3}
 \left|\Sigma_2^{(n)}(x)\right|\;\leq\; \sum_{I_n^j(c_0)\subset \mathcal{V}} |Sf(f^j(x))|\left[Df^j(x)\right]^2 \ ,
\end{equation}
Now we proceed through the following steps.
\begin{enumerate}
 \item[(i)] Since $I_n(c_0)\subset \mathcal{U}$, the sum in the right-hand side of \eqref{negS2} includes the term 
with $j=0$, namely $Sf(x)$, and by \eqref{negS0} we have
\begin{equation}\label{negS4}
 Sf(x)= -\dfrac{s_{0}^{2}-1}{2(x-c_{0})^{2}} \ .
\end{equation}
All the orther terms in \eqref{negS2} are negative as well. Since $|x-c_0|\asymp |I_n(c_0)|$ (see (ii) below), we deduce from 
\eqref{negS2} and \eqref{negS4} that
\begin{equation}\label{negS5}
 \Sigma_1^{(n)}(x)\;<\; -\frac{K_1}{|I_n(c_0)|^2} \ ,
\end{equation}
where $K_1>0$ is a constant that depends only on the real bounds and the power-law exponent $s_0$. 
\item[(ii)] Since there are no critical points of $f^{q_{n+1}}$ in $\mathrm{int}(G_i)\supset G_i^*$, the map 
$f^{q_{n+1}}: \mathrm{int}(G_i) \to f^{q_{n+1}}(\mathrm{int}(G_i))$ is a diffeomorphism{\footnote{We denote by $\mathrm{int}(X)$ the {\it interior\/} of 
the set $X\subset S^1$. }}. The same can be said of the maps $f^{j}: \mathrm{int}(G_i) \to f^{j}(\mathrm{int}(G_i))$ for $0\leq j\leq q_{n+1}-1$. 
From the real bounds and Proposition \ref{consecatoms}, we have $|f^j(G_i^*)|\asymp |f^j(G_i)|\asymp |I_n^j(c_0)|$ for all 
$0\leq j\leq q_{n+1}-1$. Moreover, both components of $G_i\setminus G_i^*$ are comparable to $G_i^*$ (hence to $I_n(c_0)$ as well). Hence, by 
Koebe's distortion principle (Lemma \ref{koebe}) and the mean-value theorem, we have
\begin{equation}\label{negS6}
 |Df^j(x)| \;\asymp\;\frac{|f^j(G_i^*)|}{|G_i^*|}\;\asymp\; \frac{|I_n^j(c_0)|}{|I_n(c_0)|}\ ,
\end{equation}
for all $x\in G_i^*$ and all $0\leq j\leq q_{n+1}-1$. 
 \item[(iii)] With step (ii) at hand, we are ready to estimate the right-hand side of \eqref{negS3}. Using \eqref{negS6}, we see that 
there exists a constant $K_2>0$ depending only on the real bounds such that 
\begin{align*}
 \left|\Sigma_2^{(n)}(x)\right|\;&\leq\; K_2\sum_{I_n^j(c_0)\subset \mathcal{V}} |Sf(f^j(x))|\left(\frac{|I_n^j(c_0)|}{|I_n(c_0)|}\right)^2 \\
 &\leq\; \frac{K_2M}{|I_n(c_0)|^2}\sum_{I_n^j(c_0)\subset \mathcal{V}} |I_n^j(c_0)|^2\ ,
\end{align*}
where $M=\sup_{y\in \mathcal{V}}|Sf(y)|<\infty$ is a constant that depends only on $f$ (and the choice of $\mathcal{U}, \mathcal{V}$). 
But 
\[
\sum_{I_n^j(c_0)\subset \mathcal{V}} |I_n^j(c_0)|^2\leq \left(\max_{I_n^j(c_0)\subset \mathcal{V}} 
|I_n^j(c_0)|\right)\sum_{I_n^j(c_0)\subset \mathcal{V}} |I_n^j(c_0)| \;<\; \delta_n \ .
\] 
Therefore
\begin{equation}\label{negS7}
 \left|\Sigma_2^{(n)}(x)\right|\; \leq\;\frac{K_2M\delta_n}{|I_n(c_0)|^2}\ .
\end{equation}
\end{enumerate}

Finally, choosing $n_0=n_0(f)>n_1$ so large that $K_2M\delta_n < K_1$ for all $n\geq n_0$, we deduce from \eqref{negS5} and \eqref{negS7} 
that, indeed, $Sf^{q_{n+1}}(x)<0$ for all $x\in G_i^*$, for all $n\geq n_0$. 
\end{proof}

\begin{figure}[ht]
\begin{center}
\psfrag{l}[][]{$\;\Delta_{k_{i+1}}$} 
\psfrag{r}[][][1]{$\Delta_{k_i}$} 
\psfrag{f}[][][1]{$\ f^{q_{n+1}}$}
\psfrag{t}[][][1]{$\cdots$}
\includegraphics[width=3.8in]{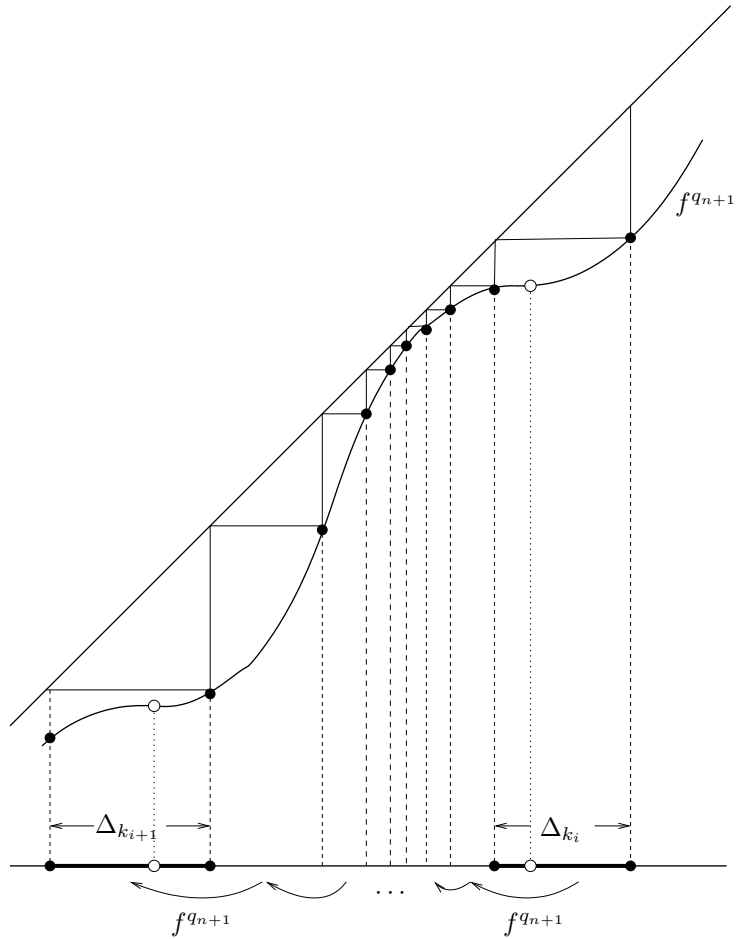}
\end{center}
\caption[doublesaddle]{\label{doublesaddle} Two consecutive critical spots and the bridge joining them: the dynamical picture.}
\end{figure}

From the lemma we have just proved, we deduce the following result concerning the bridges $G_i$, $0\leq i\leq r$, contained in the closest return 
interval $I_n(c_0)$ (see Figure~\ref{doublesaddle}). 

\begin{proposition}\label{primarybridges}
 For all $n\geq n_0$, where $n_0$ is as in Lemma \ref{negSret}, and each $i=0,1,2,\ldots,r$ for which the reduced bridge 
$G_i^*\subset I_n(c_0)$ is non-empty, the restriction 
\[f^{q_{n+1}}|_{G_i^*}\,:\,G_i^*\to f^{q_{n+1}}(G_i^*)\] 
is an almost parabolic map with length $\ell_i=k_{i+1}-k_i-3$ and width $\sigma_i\geq \sigma$, 
where $\sigma=\sigma(C)>0$ depends only on the constant $C$ in the real bounds.
\end{proposition}

\begin{proof}
 By construction, the map $\phi=f^{q_{n+1}}|_{G_i^*}$ has no critical points, hence it is a diffeomorphism onto its image. 
Since $G_i^*=\bigcup_{k=k_i+2}^{k_{i+1}-2} \Delta_k$  and $\phi(\Delta_k)=f^{q_{n+1}}(\Delta_k)=\Delta_{k+1}$ for all $k$, 
it follows that the length of $\phi$ is as stated. Moreover, by Lemma \ref{negSret}, we have $S\phi=Sf^{q_{n+1}}<0$ throughout. 
Finally, since the intervals $\Delta_{k_i+2}$ and $\Delta_{k_{i+1}-2}$ are both comparable to $G_i^*$ 
(by the real bounds and Lemma \ref{largebridge}), the last statement concerning the width of $\phi$ follows as well. 
\end{proof}

This result has the following corollary, which is the goal of the present subsection.

\begin{corollary}\label{balancedbridges}
 For all $n\in \mathbb{N}$, each non-empty bridge $G_{i,j}=f^j(G_i)\in \mathcal{P}_{n}^*(c_0)$ 
admits a balanced decomposition (with uniform comparability constants depending only on the real bounds for $f$). 
\end{corollary}

\begin{proof}
 We may of course assume that $n\geq n_0$, where $n_0$ is as in Lemma \ref{negSret}.
 For {\it primary\/} bridges, namely $G_{i,0}=G_i\subset I_n(c_0)$ ({\it i.e.,\/} those with $j=0$), the assertion follows 
from Proposition \ref{primarybridges} and Lemma \ref{balancedecomp}, together with Remark \ref{balremark}. For {\it secondary\/} bridges, 
namely $G_{i,j}=f^j(G_i)$, $1\leq j\leq q_{n+1}-1$, use the fact that $f^{j}: \mathrm{int}(G_i) \to \mathrm{int}(G_{i,j})$ is a diffeomorphism 
and apply Koebe's distortion principle (the image under $f^j$ of the balanced decomposition of $G_i$ yields a balanced decomposition 
of $G_{i,j}$, as desired). 
\end{proof}

\section{Proof of the main theorem} \label{secmainthm}

In this section, we prove our main theorem, namely Theorem \ref{maintheorem}. We follow the same strategy that was used in \cite[section 4]{dFdM} in the 
proof of the corresponding result for unicritical circle maps. At this point, all the hard work is already behind us, and the adaptation is fairly 
straightforward. 

\subsection{A criterion for quasisymmetry}

As in \cite[\S 4.2]{dFdM}, the key general tool to be used in the proof of our Main Theorem will be an extension of a result 
essentially due to Carleson \cite{C}, namely Proposition \ref{qscriterion} below. Since such tool is stated without proof in \cite{dFdM}, 
we provide a proof as a courtesy to the reader. 

First we need to recall the following definition.
\begin{definition}\label{finegrid}
A \emph{fine grid\/} is a sequence $\{\mathcal{Q}_{n}\}_{n \geq 0}$ of finite partitions of $S^{1}$ (with $\mathcal{Q}_0$ the trivial partition) 
satisfying the following conditions. 
 \begin{itemize}
  \item [(a)] Each $\mathcal{Q}_{n+1}$ is a strict refinement of $\mathcal{Q}_{n}$;
  \item [(b)] There exists an integer $a\geq 2$ such that each atom $\Delta \in \mathcal{Q}_{n}$ is the disjoint union of at most $a$ atoms of 
  $\mathcal{Q}_{n+1}$;
  \item [(c)] There exists $\rho >1$ such that $\rho^{-1} |\Delta| \leq |\Delta'| \leq \rho|\Delta|$ for each pair of adjacent atoms 
$\Delta,\Delta' \in \mathcal{Q}_{n}$. 
 \end{itemize}
The numbers $a,\rho$ are called \emph{fine grid constants\/}.
\end{definition}

We note the following consequence of the above definition: There exist $0<\alpha<\beta<1$ depending only on the fine grid constants $a,\rho$ 
such that, whenever $\Delta\in \mathcal{Q}_{n}$, $\Delta^*\in \mathcal{Q}_{n-1}$ and $\Delta\subset \Delta^*$, we have
\begin{equation}\label{finegridconseq}
 \alpha|\Delta^*|\leq |\Delta|\leq \beta |\Delta^*|\ .
\end{equation}
In fact, it is not difficult to see that one can take $\alpha=(a\rho^{a-1})^{-1}$ and $\beta=(1+\rho^{-1})^{-1}$. 

\begin{proposition}\label{qscriterion} Let $\{\mathcal{Q}_{n}\}_{n \geq 0}$ be a fine grid in $S^1$, with fine grid constants $a,\rho$, and 
let $h : S^{1} \rightarrow S^{1}$ be a homeomorphism such that 
  \begin{equation}\label{fineh}
   \left| \dfrac{|\Delta'|}{|\Delta''|} - \dfrac{|h(\Delta')|}{|h(\Delta'')|}  \right| \leq \lambda,
  \end{equation}
for each pair of adjacent atoms $\Delta',\Delta'' \in \mathcal{Q}_{n} $, for all $n\geq 0$, where $\lambda>0$ is a given constant.
Then there exists $K=K(a,\rho,\lambda)>1$ such that $h$ is $K$-quasisymmetric.
 \end{proposition} 

We stress that the image of a fine grid under a homeomorphism satisfying the set of conditions \eqref{fineh} is also a fine grid. 
Conversely, if a homeomorphism $h$ is an {\it isomorphism\/} between two fine grids, {\it i.e.\/} if $h$ establishes 
a perfect correspondence between the atoms of both, then $h$ must satisfy a set of conditions like \eqref{fineh}.    
 
In order to prove Proposition \ref{qscriterion}, we need the following auxiliary lemma. 

\begin{lemma}\label{auxlemqs}
 Given a fine grid $\{\mathcal{Q}_{n}\}_{n \geq 0}$ with fine grid constants $a,\rho$ as above, let $I\subset S^1$ be an interval 
with non-empty interior, and let $n=n(I)$ 
be the smallest natural number such that $I\supset \Delta$  for some atom $\Delta\in \mathcal{Q}_n$. Then there exists an interval 
$U\supset I$ with the following properties:
\begin{enumerate}
 \item[(i)] $U$ is the union of at most $2a$ atoms of $\mathcal{Q}_n$;
 \item[(ii)] $|U|\leq \alpha^{-1}(1+\rho)|I|$, where $\alpha$ is the constant in \eqref{finegridconseq}. 
\end{enumerate}
\end{lemma}

\begin{proof}
Suppose $I$ intersects $3$ distinct consecutive atoms of $\mathcal{Q}_{n-1}$, say $\Delta_1,\Delta_2,\Delta_3$, with $\Delta_2$ lying between 
$\Delta_1$ and $\Delta_3$. Then we necessarily have $\Delta_2\subseteq I$; but this contradicts the definition of $n=n(I)$. Hence $I$ is contained 
in the union $U$ of at most two atoms of $\mathcal{Q}_{n-1}$. Since each atom of $\mathcal{Q}_{n-1}$ is the union of at most $a$ atoms of  
$\mathcal{Q}_{n}$, part (i) follows. To prove (ii), given that $I\supset \Delta\in \mathcal{Q}_{n}$, let $\Delta^*$ be the unique atom of 
$\mathcal{Q}_{n-1}$ that contains $\Delta$. By part (i), $U$ contains $\Delta^*$ and at most one other atom $\Delta^{**}\in \mathcal{Q}_{n-1}$ 
adjacent to $\Delta^*$. Therefore, using property (c) in Definition \ref{finegrid} and \eqref{finegridconseq}, we have
\[
 |U|\leq |\Delta^*|+|\Delta^{**}| \leq (1+\rho)|\Delta^*|\leq \alpha^{-1}(1+\rho)|\Delta|\leq \alpha^{-1}(1+\rho)|I|\ .
\]
This establishes (ii) and finishes the proof.
\end{proof}

\begin{proof}[Proof of Proposition \ref{qscriterion}]
We will verify the quasisymmetry condition 
\[
 \frac{1}{K}\leq \frac{h(x+t)-h(x)}{h(x)-h(x-t)}\leq K 
\]
for all $x\in S^1=\mathbb{R}/\mathbb{Z}$ and all $t>0$, with $K>1$ a constant to be determined in the course of the argument.
Let $I=[x-t,x+t]$ be the interval on the circle that contains $x$, and write $I=I^{-}\cup I^{+}$, where $I^{-}=[x-t,x]$ and 
$[x,x+t]$. 
By Lemma \ref{auxlemqs}, there exist $n=n(I)$ and an 
interval $U\supset I$ such that $U$ is the union of at most $2a$ atoms of $\mathcal{Q}_n$  and $|U|\leq \rho_1|I|$, where 
$\rho_1=\alpha^{-1}(1+\rho)$. Let $p$ be the smallest positive integer such that $\beta^p\rho_1<\frac{1}{4}$. Write 
$U$ as the union of atoms of $\mathcal{Q}_{n+p}$, say
\[
 U=J_1\cup J_2\cup\cdots \cup J_s\ ,
\]
where the $J_i\in \mathcal{Q}_{n+p}$, $1\leq i\leq s$ are assumed to be ordered counterclockwisely on the circle. Note that 
we must have $s\leq 2a^{p+1}$. By \eqref{finegridconseq} and induction, we have $|J_i|\leq \beta^{p}|J_i^*|$, where 
$J_i^*\subseteq U$ is the unique atom of $\mathcal{Q}_n$ that contains $J_i$. Hence we get
\[
|J_i|\leq \beta^p|J_i^*|\leq \beta^p|U|\leq \beta^p\rho_1|I|<\frac{1}{4}|I|\ .
\] 
But this means that at least one of the $J_i$'s, say $J_{i_0}$, is contained in $I^{-}$. 
Thus, we have on the one hand $J_{i_0}\subset I^{-}$ and on the other hand 
$I^{+}\subseteq J_{i_0+1}\cup J_{i_0+2}\cup \cdots \cup J_s$. Moreover, by the hypothesis \eqref{fineh}, for all $1\leq i\leq s-1$ 
we have
\[
 \frac{|h(J_{i+1})|}{|h(J_i)|}\leq \lambda + \frac{|J_{i+1}|}{|J_i|}\leq \lambda + \rho\ ,
\]
from which it follows by telescoping that 
\[
 \frac{|h(J_{i+\nu})|}{|h(J_i)|}\leq (\lambda+\rho)^\nu\ \ \textrm{for all}\ \nu=1,2,\cdots,s-i\ .
\]
Therefore
\begin{align*}
 \frac{h(x+t)-h(x)}{h(x)-h(x-t)} &= \frac{|h(I^{+})|}{|h(I^{-})|} \leq \dfrac{\sum_{i=i_0+1}^{s} |h(J_i)|}{|h(J_{i_0})|}\\
 & \leq \sum_{\nu=1}^{s-{i_0}} (\lambda+\rho)^\nu \leq \sum_{\nu=1}^{2a^{p+1}} (\lambda+\rho)^\nu \ .
\end{align*}
This proves that $h$ is $K$-quasisymmetric with $K= \sum_{\nu=1}^{2a^{p+1}} (\lambda+\rho)^\nu$, a constant that indeed 
depends only on the constants $a,\rho, \lambda$. 
\end{proof}

\subsection{A suitable fine grid}\label{sec:finegrid}

Now we define an auxiliary partition $\widetilde{\mathcal{P}}_n^*(c_0)$, for each
$n\ge 1$. The atoms of $\widetilde{\mathcal{P}}_n^*(c_0)$  are all atoms of 
$\mathcal{P}_n^*(c_0)$ which are not saddle-node, together with the atoms of the
balanced partitions of all saddle-node atoms of $\mathcal{P}_n^*(c_0)$. The
partition $\mathcal{Q}_n$ that we want is constructed from $\widetilde{\mathcal{P}}_m^*(c_0)$ and 
$\mathcal{P}_m^*(c_0)$ for various values of $m\leq n$ as follows.

\begin{proposition} \label{gridpprop}
There exists a fine grid $\{\mathcal{Q}_n\}$ in
$S^1$ with the following properties.
\begin{enumerate}
\item[($a$)] Every atom of $\mathcal{Q}_n$ is the union of at most $a=4N+3$ 
atoms{\footnote{As we saw in \S \ref{auxpart}, each long interval $I_n^i(c_0)\in \mathcal{P}_n(c_0)$ is decomposed as the union of $2r+3\leq 4N+3$ atoms 
of $\mathcal{P}_n^*(c_0)$.}} of $\mathcal{Q}_{n+1}$.
\item[($b$)] Every atom $\Delta\in \mathcal{Q}_n$ is a union of atoms of
$\mathcal{P}_m^*(c_0)$ for some $m\leq n$, and there are three possibilities:
\begin{enumerate}
\item[($b_1$)] $\Delta$ is a single atom of $\mathcal{P}_m^*(c_0)$;
\item[($b_2$)] $\Delta$ is a central interval of $\widetilde{\mathcal{P}}_m^*(c_0)$; 
\item[($b_3$)] $\Delta$ is the union of at least two atoms of
$\mathcal{P}_{m+1}^*(c_0)$ contained in a single atom of $\widetilde{\mathcal{P}}_m^*(c_0)$.
\end{enumerate}
\end{enumerate}
\end{proposition}

\begin{proof} The proof is by induction on $n$. The first partition $\mathcal{Q}_1$ 
consists of all atoms of $\mathcal{P}_1^*(c_0)$ which are not saddle-node atoms
together with the intervals $L_0$, $M_1$ and $R_0$ of each saddle-node
interval $I\in \mathcal{P}_1^*(c_0)$ ($I=L_0\cup M_1\cup R_0$). It is clear that each
atom of $\mathcal{Q}_1$ falls within one of the categories ($b_1$)-($b_3$)
above. 

Assuming $\mathcal{Q}_n$ defined, define $\mathcal{Q}_{n+1}$ as follows. Take an
atom $I\in \mathcal{Q}_n$ and consider the four cases below.
\begin{enumerate}
\item[(1)] If $I$ is a single atom of $\mathcal{P}_m^*(c_0)$ then one of two things can
happen:
\begin{enumerate}
\item[(i)] $I$ is a saddle-node atom: In this case write $I=L_0\cup
M_1\cup R_0$ as above and take $L_0$, $R_0$ and $M_1$ as atoms of $\mathcal{Q}_{n+1}$. 
Note that the lateral intervals $L_0$ and $R_0$ are atoms of type
($b_1$), while the central interval $M_1$ is of type ($b_2$).
\item[(ii)] $I$ is not a saddle-node atom: Here, there are two sub-cases to consider. 
The first possibility is that $I$ is a single (regular) atom of 
$\mathcal{P}_m(c_0)$, in which case we break it into the union of at most $a$ atoms of 
$\mathcal{P}_{m+1}^*(c_0)$ and take them as atoms of $\mathcal{Q}_{n+1}$, all of which are 
of type ($b_1$). The second possibility is that $I$ is a (short) bridge, in which case 
we break it up into its $\leq 1,000$ constituent atoms of $\mathcal{P}_{m+1}(c_0)$ 
and take them as atoms of $\mathcal{Q}_{n+1}$, again all of type ($b_1$). 
\end{enumerate}
\item[(2)] If $I$ is a central interval of $\widetilde{\mathcal{P}}_m^*(c_0)$ 
which is not the final interval, consider the next central interval of $\widetilde{\mathcal{P}}_m^*(c_0)$ 
inside $I$, say $M$, and the two corresponding lateral intervals $L$ and $R$
such that $I=L\cup M\cup R$, and declare $L$, $R$ and $M$ members of $\mathcal{Q}_{n+1}$. 
Note that $L$ and $R$ are of type ($b_3$), while $M$ is of type
($b_2$).

\item[(3)] If $I$ is a union of $p\ge 2$ consecutive atoms
$J_1,\ldots,J_p$ of $\mathcal{P}_{m+1}(c_0)$ inside a single atom of
$\mathcal{P}_m^*(c_0)$ (this happens when $I$ is contained in a lateral interval of the balanced decomposition of a long bridge), 
divide it up into two approximately equal parts. More
precisely, write $p=2q+r$, where $r=0$ or $1$, and consider $I=L\cup R$
where 
$$
L\;=\;\bigcup_{j=1}^qJ_j \ ,
\ R\;=\;\bigcup_{j=q+1}^pJ_j
\ .
$$
We obtain in this fashion two new
atoms of $\mathcal{Q}_{n+1}$ (namely $L$ and $R$) which are either single atoms of 
$\mathcal{P}_{m+1}(c_0)$, and therefore of type ($b_1$), or once again intervals 
of type ($b_3$).
\end{enumerate}

This completes the induction. That $\{\mathcal{Q}_{n}\}_{n\ge 1}$ constitutes a
fine grid follows easily from the real bounds, Lemma \ref{balancedecomp}, Remark \ref{ratioofJs} 
and Corollary \ref{balancedbridges}. Indeed, it suffices to verify that condition (c) of Definition \ref{finegrid} is 
satisfied (for some constant $\rho>1$ depending only on the real bounds). 
Given two ajacent atoms $\Delta,\Delta'\in \mathcal{Q}_n$, there are two cases to consider.
\begin{enumerate}
 \item[(a)] There exist $m, m'\leq n$ such that $\Delta$ is a single atom of $\mathcal{P}_m(c_0)$ and $\Delta'$ is a single atom of $\mathcal{P}_{m'}(c_0)$. 
In this case, either $m=m'$, or $m$ and $m'$ differ by 1 (this is easily proved by induction 
on $n$ from the construction of $\mathcal{Q}_n$ given above). But then we have $|\Delta|\asymp |\Delta'|$ by the real bounds (Theorem 
\ref{realbounds}). 
 \item[(b)] For some $m\leq n$, at least one of the two atoms, say $\Delta$, is the union of $p\geq 2$ atoms of $\mathcal{P}_{m+1}(c_0)$ inside a single atom 
of $\mathcal{P}_m^*(c_0)$, which is necessarily a bridge. This implies that both $\Delta$ and $\Delta'$ are contained in the same 
bridge $G\in \mathcal{P}_m^*(c_0)$. Looking at the balanced decomposition of $G$ (given by Corollary \ref{balancedbridges}), 
we see that there are two possibilities. The first possibility is that both $\Delta$ and $\Delta'$ are contained in the same lateral interval 
($L_i,R_i$) or the same central interval ($M_i$) of said balanced decomposition. In this case, $\Delta$ and $\Delta'$ are both unions of the 
{\it same\/} number of fundamental domains of $G$, and we have $|\Delta|\asymp |\Delta'|$ by Lemma \ref{balancedecomp} and Remark \ref{ratioofJs}. 
The second possibility is that $\Delta$ and $\Delta'$ are contained in adjacent intervals of the balanced decomposition of $G$. In this case, 
one of the two atoms, $\Delta$ or $\Delta'$, is the union of at most twice as many fundamental domains of $G$ as the other, and 
we have $|\Delta|\asymp |\Delta'|$, again by Lemma \ref{balancedecomp} and Remark \ref{ratioofJs}.
\end{enumerate}
This establishes the desired comparability of adjacent atoms of $\mathcal{Q}_n$ in all cases, with uniform constants depending only on the real bounds, 
and the proof is complete.
\end{proof} 

\subsection{Proof of Theorem \ref{maintheorem}} Everything is in place now, so we can give the proof of the main theorem 
in just a few lines.

\begin{proof}[Proof of Theorem \ref{maintheorem}]
 Consider the fine grids $\{\mathcal{Q}_n(f)\}$ and $\{\mathcal{Q}_n(g)\}$ given by Proposition \ref{gridpprop} applied to 
$f$ and $g$, respectively (the construction being based on chosen critical points $c_0(f), c_0(g)$ with $h(c_0(f))=c_0(g)$). 
Since by hypothesis the conjugacy $h$ between $f$ and $g$ maps each critical point of $f$ to 
a corresponding critical point of $g$, it follows that  $h$ maps each critical spot of $\mathcal{P}_m^*(c_0(f))$ 
to a corresponding critical spot of $\mathcal{P}_m^*(c_0(g))$, and likewise for bridges, for all $m\geq 1$. 
This means that $h$ is an isomorphism between the two fine grids. Therefore $h$ is quasisymmetric, by 
Proposition \ref{qscriterion}. 
\end{proof}

\section{Final comments}\label{sec:final}

The quasisymmetric rigidity theorem we have just proved assumes that the conjugacy $h$ between 
$f$ and $g$ maps the critical points of $f$ to the critical points of $g$, but it does not assume that 
the critical exponents at corresponding critical points are the same. This is clearly a necessary condition 
for $h$ to be $C^1$. We conjecture that this condition is also sufficient. 

Let us be more precise. Given a multicritical circle map $f$ with $n_f$ critical points $c_i=c_i(f)$, $0\leq i\leq n_f-1$,
and irrational rotation number $\rho(f)$, let $\mu_f$ be its unique invariant Borel probability measure. 
We define the {\it signature\/} of $f$ to be the $(2n_f+2)$-tuple 
\[
(\rho(f),n_f;s_0,s_1,\ldots,s_{n_f-1};\lambda_0,\lambda_1,\ldots,\lambda_{n_f-1}) \ ,
\]
where  $s_i$ is the critical exponent of the critical point $c_i$ and 
$\lambda_i=\mu_f[c_i,c_{i+1})$ (with the convention that $c_{n_f}=c_0$). Note that $\sum_{i=0}^{n_f-1}\lambda_i=1$, and 
the number of critical points $n_f$ is redundant information once we are given the rest of the data. Hence  
the signature of $f$ carries, in fact, only $2n_f$ independent parameters.  

\begin{conjecture}\label{conjecture}
 If two $C^3$ multicritical circle maps without periodic points have the same signature, then they are conjugate by a 
$C^1$ diffeomorphism. Moreover, if their common rotation number is of \emph{bounded type\/}, then the conjugacy is in fact 
$C^{1+\alpha}$ for some universal $\alpha>0$.
\end{conjecture}

This conjecture is formulated as an extension of the corresponding conjecture for critical circle maps with a {\it single\/} 
critical point (in which case the rotation number and the unique critical exponent are the only invariants). In that 
context, the conjecture has been proven in the real-analytic case thanks to the combined efforts of de Faria \cite{dF}, de Faria and de Melo \cite{dFdM,dFdM2}, 
Yampolsky \cite{Yam1,Yam2,Yam} and Khanin and Teplinsky \cite{KT}. Note that in the real-analytic case the critical exponent is necessarily an 
odd integer. Still in the unicritical case, Guarino has shown in his thesis \cite{Gua} (see also \cite{GdM}) that the second part 
of the above conjecture holds true in the $C^3$ category provided the critical exponent is, again, an odd integer. 
For non-integer critical exponents, the conjecture is wide open even in the unicritical case. 
One cannot expect the conjugacy to be $C^{1+\alpha}$ (with positive $\alpha$) for arbitrary rotation numbers. This is shown via (unicritical) examples constructed 
by de Faria and de Melo \cite{dFdM} in the $C^{\infty}$ (or $C^k$) case, and by Avila \cite{Av} in the real analytic case. 
These examples help clarify, to a certain extent, the limits of validity of the above conjecture. 

It is reasonable to expect that a proof of Conjecture \ref{conjecture} will 
require the development of a complete renormalization theory for multicritical circle maps, perhaps including 
an analogue of the concept of holomorphic commuting pair (as defined in \cite{dF}). These matters will be pursued in a forthcoming paper. 

\section*{Acknowledgements}
We are grateful to Pablo Guarino and Charles Tresser for several useful conversations on these and related matters. We wish to thank 
also Sofia Trejo for an interesting discussion leading to the correct formulation of Lemma \ref{criticlarge}.


\end{document}